\documentclass[11pt]{article}
\usepackage{verbatim} 
\usepackage{graphicx}
\usepackage[utf8]{inputenc}
\usepackage{amsmath,amsfonts,amssymb,amscd,amsthm,txfonts,hyperref}
\usepackage{xcolor}
\newtheorem{theorem}{Theorem}[section]
\newtheorem{proposition}[theorem]{Proposition}
\newtheorem{lemma}[theorem]{Lemma}
\newtheorem{corollary}[theorem]{Corollary}
\newtheorem{remark}{Remark}[section]
\newtheorem{theo}{Theorem}

\theoremstyle{definition}
\newtheorem{definition}[theorem]{Definition}

\newcommand{\R}{\mathbb{R}}
\newcommand{\C}{\mathbb{C}}
\newcommand{\N}{\mathbb{N}}
\newcommand{\Z}{\mathbb{Z}}
\newcommand{\T}{\mathbb{T}}
\renewcommand{\S}{\mathbb{S}}

\newcommand{\lap}{\bigtriangleup}

\newcommand{\E}{\mathbb E}
\renewcommand{\Re}{\mbox{Re}}

\newcommand{\mE}{\mathcal E}

\newcommand{\an}[1]{\langle #1 \rangle}
\newcommand{\parmi}[2]{\begin{pmatrix}#2 \\ #1\end{pmatrix}}
\newcommand{\grad}{\bigtriangledown}
\renewcommand{\L}{\mathcal L}
\renewcommand{\Im}{\textrm{Im}}

\title{An equation on random variables and systems of fermions}
\author{Anne-Sophie de Suzzoni\footnote{Universit\'e Paris 13, Sorbonne Paris Cit\'e, LAGA, CNRS ( UMR 7539), 99, avenue Jean-Baptiste Cl\'ement, F-93430 Villetaneuse, France - email: adesuzzo@math.univ-paris13.fr}}

\begin{document}

\maketitle

\begin{abstract} In this paper, we consider an equation on random variables which can be reduced to the equation which describes the evolution of systems of fermions. We give some results of well-posedness for this equation on the spheres and torus of dimension 2 and 3  and on the Euclidean space. We give results of scattering and blow-up on the Euclidean depending on if the equation is defocusing or focusing. We interpret the results in terms of the evolution of fermions.
\end{abstract}

\tableofcontents

\section{Motivations}

In this paper, we present an equation on random variables related to systems of fermions. This section is dedicated to presenting this equation and explaining its relation to equations derived from many-body quantum physics.  We consider that, under sufficient assumptions, a system of fermions should behave according to 
$$
i\partial_t \gamma = [-\lap + w* \rho_\gamma, \gamma]
$$
where $\gamma$ is a non negative bounded integral operator with kernel $\gamma(y,x)$, where $\rho_\gamma$ is the multiplication by $\gamma(x,x)$, and $[\cdot,\cdot ]$ is the commutator. The map $w$ may be a Dirac delta. This equation has been studied in \cite{eq1,eq2,eq3,lewsabI,lewsabII,eq4}.

The interest is that the equation on random variable closely resembles the cubic Schr\"odinger equation, and the theory of Schr\"odinger equations only has to be adapted to random variables to provide results, which are eventually turned into properties for the systems of fermions.

In Sections \ref{sec-WPST}, \ref{sec-WPE}, \ref{sec-scatt}, \ref{sec-focus}, we use previously existing techniques about the cubic Schr\"odinger equation and adapt them to random variables. In Section \ref{sec-incidence}, we give and discuss corollaries of the previous sections for systems of fermions.

\subsection{Dynamics of a system of fermions}

Before describing the dynamics of a system of fermions, we start with the better known Bose-Einstein condensate. 

A system of $N$ bosons may be described by a wave function $\Psi(x_1,\hdots, x_N)$. from $\R^{3N}$ to $\C$. It satisfies under certain conditions the Schr\"odinger equation
$$
i\partial_t \Psi = -\sum_{i=1}^N \lap_{x_i} \Psi +\sum_{i\neq j} w_T(x_i-x_j) \Psi
$$
where $\lap_{x_i}$ is the laplacian with respect to the variable $x_i$ and is related to the kinetic energy, and $w_T$ is related to the interaction between particles and depends on the temperature $T$.

When one lowers the temperature and takes a large number of particles, the system becomes a Bose-Einstein condensate, and under a mean-field approximation, one writes $\Psi(x_1,\hdots , x_N) = \prod_j u(x_j)$ with $u$ satisfying an equation of the form :
$$
i\partial_t u =-\lap u +w* |u|^2 u.
$$
This approximation is motivated by the fact that bosons are exchangeable particles, in the sense that $\Psi$ is symmetric, that is 
$$
\Psi(x_{\sigma(1)},\hdots, x_{\sigma(N)}) = \Psi(x_1,\hdots, x_N)
$$
for all permutations $\sigma$. The derivation of Bose-Einstein dynamics from many-body quantum mechanics is a vast subject in the literature, see for instance \cite{derbos5,derbos6,derbos8,derbos3,derbos1,derbos2,derbos4,derbos7}.

Let us now consider a system of fermions. It is described by a wave function $\Psi$ satisfying the same kind of dynamics as a system of bosons. But since we are dealing with fermions, $\Psi$ is anti-symmetric, that is
$$
\Psi(x_{\sigma(1)},\hdots, x_{\sigma(N)}) = \varepsilon(\sigma)\Psi(x_1,\hdots, x_N)
$$
where $\varepsilon (\sigma)$ is the signature of the permutation $\sigma$. This is the Pauli principle. If one writes 
$$
\Psi(x_1,\hdots, x_N)= \frac1{\sqrt{n!}}\sum_{\sigma } \varepsilon(\sigma) \prod_{j=1}^N u_{\sigma (j)}(x_j)
$$
where $u_j$ are orthonormal functions, then the dynamics of $\Psi$ may be approached, under a mean-field approximation, by the Hartree-Fock equation :
$$
\forall j=1,\hdots, N \; , \; i\partial_t u_j = -\lap u_j + w*(\sum_k |u_k|^2) u_j.
$$
Note that $\int \overline{u_k}u_j$ is a conserved quantity for this equation and hence the orthonormality is preserved under the flow. The derivation of the Hartree-Fock equation from many-body quantum mechanics may be found in \cite{derfer1,derfer2,derfer3,derfer4,derfer5}.

Writing $\gamma = \sum_k |u_k \times u_k|$, where $|f\times g|$ is the operator such that
$$
|f\times g|(v) (x) = \int \overline{g(y)}v(y) dy f(x),
$$
we get that $\gamma$ satisfies 
$$
i\partial_t \gamma  = [-\lap + w* \rho_\gamma , \gamma]
$$
where $[\cdot , \cdot]$ is the commutator and $\rho_\gamma$ is the diagonal of the integral kernel of $\gamma$, here $\rho_\gamma = \sum |u_k|^2$. We note that the number of particles $N$ is equal to the trace of $\gamma$. One may consider this equation on self-adjoint integral operators $\gamma$ such that $0\leq \gamma \leq 1$. These are called density operators. One can then consider a more general setting for the systems of fermions. For instance, by not restricting $\gamma$ to be a trace-class operator, one can consider infinite systems of particles. The stability of non-trace class stationary solutions is the subject of \cite{lewsabI,lewsabII}, which inspired this paper.

\subsection{Comparison with density operators}

We present here the equation on random variables and explain how it is related to what has been said before.

We consider the equation on random variables : 
\begin{equation}\label{eqonrv}
i\partial_t X = - \lap X + \E (|X|^2) X
\end{equation}
on a probability space $(\Omega, \mathcal A ,P)$. We assume that $X$ has values in $L^2_{\textrm{loc}}(M)$ where $M$ is either $\S^d$, $\T^d$ or $\R^d$.

We write 
$$
\an{f,g} = \int_{M} \overline{f(x)}g(x)dx.
$$

\begin{proposition}\label{XtoG} Let $\gamma$ be the operator defined as 
$$
\gamma  = \int |X(\omega) \times X(\omega)| dP(\omega)
$$
that is 
$$
\gamma (v) = \E (\an{X,v} X).
$$
Let $\rho_\gamma$ be the diagonal of the integral kernel of $\gamma$. Then, $\gamma$ solves the equation:
\begin{equation}\label{eqonop}
i\partial_t \gamma = [-\lap + \rho_\gamma, \gamma] .
\end{equation}
\end{proposition}

\begin{remark} This is the equation one can find in \cite{lewsabI,lewsabII} in the case $\omega = \delta$. 
\end{remark}

\begin{proof} Let $v$ in the domain of definition of $\gamma$ and let us differentiate $\gamma (v)$. We have
$$
i\partial_t \gamma(v) = \E (\an{-i\partial_t X, v} X) + \E(\an{X,v} i\partial_t X)
$$
and by replacing $i\partial_t X$ by its value, we get
$$
i\partial_t \gamma(v) = \E (\an{ \lap X - \E (|X|^2) X, v} X) + \E(\an{X,v} (- \lap X + \E (|X|^2) X)).
$$
Because $\lap$ and the multiplication by $\E(|X|^2)$ are self-adjoint, we get
$$
\E (\an{ \lap X - \E (|X|^2) X, v} X) = \E (\an{  X, (\lap - \E(|X|^2)v} X) = \gamma((\lap - \E(|X|^2))v).
$$
As $\an{X,v} $ depends only on the probability variable, we have 
$$
\an{X,v} (- \lap X + \E (|X|^2) X)= (-\lap + \E(|X|^2))\Big( \an{X,v} X)\Big)
$$
and since $-\lap + \E(|X|^2)$ does not act on the random variable, 
$$
\E \Big (-\lap + \E(|X|^2))\Big( \an{X,v} X) \Big)\Big)= (-\lap + \E(|X|^2))\E (\an{X,v}X) = (-\lap + \E(|X|^2))(\gamma (v))
$$
therefore
$$
i\partial_t \gamma(v) = [-\lap + \E(|X|^2), \gamma].
$$
What is more, the integral kernel of $\gamma$ is $\E(\overline X(y) X(x))$ and hence $\rho_\gamma (x) = \E(|X(x)|^2)$ which gives the result.
\end{proof}

This proposition explains how one goes from a solution of \eqref{eqonrv} to a solution of \eqref{eqonop}. The following proposition explains how to pass from an initial datum for \eqref{eqonop} to an initial datum for \eqref{eqonrv}. Combining these two propositions and a global well-posedness property for \eqref{eqonrv}, we get global existence for \eqref{eqonop}. Indeed, from an initial datum for \eqref{eqonop}, we get an initial datum for \eqref{eqonrv}, which gives a global solution to \eqref{eqonrv}, which is turned into a solution to \eqref{eqonop}..

\begin{proposition}\label{prop-CI}Let $s\geq 0$. Let $\gamma_0 $ be a non negative trace class operator on $L^2(M)$ such that 
$$
\textrm{Tr} ((1-\lap)^s\gamma_0) < \infty.
$$
There exists a probability space $(\Omega, \mathcal A,P)$ and a random variable on this space $X_0$ such that $X_0 \in L^2(\Omega, H^s(M))$ and for all $v \in L^2(M)$
$$
\E(\an{X_0,v}X_0) = \gamma_0 (v).
$$
\end{proposition}

\begin{proof} As $\gamma_0$ is trace class and non-negative, there exists a sequence of non-negative numbers $(\alpha_n)_{n\in \N}$ and an orthonormal family of $L^2(M)$, $(e_n)_{n\in\N}$ such that
$$
\gamma_0 = \sum_{n\in \N} \alpha_n |e_n \times e_n|
$$
where $|e_n \times e_n|$ is the projection on $\C e_n$.

Let $(g_n)_{n\in N}$ be a sequence of complex centred normalised independent Gaussian variables. Set
$$
X_0 = \sum_{n\in \N} \sqrt{\alpha_n} g_n e_n.
$$

Let $v\in L^2(M)$. We have
$$
\E(\an{X_0,v}X_0) = \sum_{k,l} \sqrt{ \alpha_k \alpha_l}\an{e_k,v} e_l \E(\overline g_k g_l)
$$
and since $\E(\overline g_k g_l) = \delta_k^l$ where $\delta_k^l$ is the Kronecker symbol, we get
$$
\E(\an{X_0,v}X_0) = \sum_{k}  \alpha_k \an{e_k,v} e_k = \gamma_0(v).
$$

Besides, we have by definition
$$
\|X_0\|_{L^2(\Omega, H^s)}^2= \E(\an{X_0,(1-\lap)^sX_0})
$$
and since $\textrm{Tr}(AB) = \textrm{Tr}(BA)$,
$$
\|X_0\|_{L^2(\Omega, H^s)}^2= \E\Big(\textrm{Tr}(|X_0\times X_0| (1-\lap)^s)\Big)
$$
and by linearity of the trace and definition of $X_0$,
$$
\|X_0\|_{L^2(\Omega, H^s)}^2= \textrm{Tr}(\E(|X_0\times X_0|) (1-\lap)^s) = \textrm{Tr}(\gamma_0 (1-\lap)^s).
$$
\end{proof}

\begin{remark}
More generally, if $\gamma_0$ is a non-negative operator and $X_0$  is the Gaussian random field (see \cite{simonPphi2}) with covariance operator $\gamma_0$ then $\gamma_{X_0} = \gamma_0$. \end{remark}

\subsection{Equilibria}

The equation \eqref{eqonop} has stationary states on $\R^d$, $\T^d$, and $\S^d$, or even on sufficiently symmetric spaces. By sufficiently symmetric spaces, we mean any manifold $M$ such that there exists a transitive action of a group on $M$ that leaves $M$ invariant. 

On $\R^d$ and $\T^d$, all Fourier multipliers may be considered. Indeed, they commute with the Laplacian and their integral kernel is a function of $x-y$, making their diagonals constants, hence commuting with any operator.

On $\S^d$, one may consider functions of the Laplace-Beltrami operator. These operators commute with the Laplace-Beltrami operator and the diagonal of their kernels is also a constant. This is due to spherical symmetry and is explained later.

In this subsection, we present random variables related to these stationary states. What we obtain from this parallel are not stationary states but states whose laws are invariant under the flow of \eqref{eqonrv}.

\paragraph{On the sphere $\S^d$}	

For $n\in \N^*$, let $(e_{n,k})_{1\leq k\leq N_n}$ be a $L^2$ basis of spherical harmonics of degree $n$, that is, $e_{n,k}$ satisfies
$$
-\lap_{\S^d} e_{n,k} = n(n+d-1) e_{n,k} = \lambda_n e_{n,k}
$$
for all $k = 1,\hdots, N_n$. The number $N_n$ is the dimension of the spherical harmonics of degree $n$, it is equal to
$$
N_n = \parmi{d}{n+d} - \parmi{d}{n+d-2} \sim \frac2{(d-1)!}n^{d-1}.
$$

Let $(a_n)_{n\geq 1}$ be a sequence of complex numbers satisfying
$$
\sum_{n\geq 1} n^{d+1} |a_n|^2 < \infty.
$$

Let $(g_{n,k})_{n,k}$ be a sequence of independent complex Gaussian variables of law $\mathcal N(0,1)$.

We set
$$
Y_0 = \sum_{n,k} g_{n,k} a_n e_{n,k} \mbox{ and } m= \frac1{\textrm{vol}(\S^d)}\sum_{n>0} N_n |a_n|^2
$$
and finally
$$
Y(t) = \sum_{n,k}  g_{n,k} a_n e^{-it(\lambda_n + m)}e_{n,k}.
$$

\begin{proposition}\label{prop-invar} The random variable $Y(t)$ satisfies \eqref{eqonrv} and its law does not depend on $t$. Besides $Y$ belongs to $L^2(\Omega, H^1(\S^d))$.\end{proposition}

\begin{remark} Even though $Y$ is not a stationary solution, this makes $Y$ a natural invariant or equilibrium for \eqref{eqonrv}. \end{remark}

To prove this proposition, we need the following lemma.

\begin{lemma}[\cite{SteinWeiss},Lemma 3.1 in \cite{burlebInj}] The quantity
$$
K_n(x) = \sum_{k=1}^{N_n} |e_{n,k}(x)|^2
$$
does not depend on $x$ and is equal to
$$
\frac{N_n}{\textrm{vol}(\S^d)}.
$$
\end{lemma}

As this lemma is crucial for the invariance of $Y$, we give some elements of its proof.

\begin{proof}Let 
$$
\tilde K_n(x,y) = \sum_{k=1}^{N_n}e_{n,k}(x)\overline{e_{n,k}(y)}
$$
be the integral kernel of the orthogonal projection on the spherical harmonics of degree $n$. Because the sphere is invariant under rotations, we have for every rotation $R$ that $(e_{n,k}\circ R)_{1\leq k\leq N_n}$ is also a $L^2$ orthonormal basis of the spherical harmonics of degree $n$. Hence, $\tilde K_n(Rx,Ry)$ is also the integral kernel of the orthogonal projection on the spherical harmonics of degree $n$. Thus, for all rotations $R$ and all $x\in \S^d$ 
$$
K_n(Rx) = \tilde K_n(Rx,Rx) = \tilde K_n(x,x) = K_n(x).
$$
Let $x_0 \in \S^d$. For all $x\in \S^d$, there exists a rotation $R$ such that $x = Rx_0$, hence
$$
K_n(x) = K_n(Rx_0) = K_n(x_0)
$$
and $K_n(x)$ does not depend on $x$.

Finally
$$
K_n(x_0) = \frac1{\textrm{vol}(\S^d)}\int_{\S^d} K_n(x_0) dx =  \frac1{\textrm{vol}(\S^d)}\int_{\S^d} K_n(x) dx .
$$
And given the definition of $K_n$ and the fact that $(e_{n,k})_{1\leq k\leq N_n}$ is an orthonormal basis, we have 
$$
K_n(x_0) = \frac1{\textrm{vol}(\S^d)}\int_{\S^d} \sum_{k=1}^{N_n}|e_{n,k}(x)|^2 dx =\frac{N_n}{\textrm{vol}(\S^d)}
$$
which concludes the proof.\end{proof}

\begin{proof}[Proof of Proposition \ref{prop-invar}.]

Let us compute $\E(|Y(x)|^2)$. Because of the independence of the Gaussian variables, we have
$$
\E(|Y(x)|^2) = \sum_{n,k} |a_n|^2 |e_{n,k}(x)|^2.
$$
We use the lemma to get
$$
\E(|Y(x)|^2) = \sum_n |a_n|^2 K_n(x) = \sum_n |a_n|^2\frac{N_n}{\textrm{vol}(\S^d)}=m.
$$

We differentiate $Y$. We get
$$
i\partial_t Y = \sum_{n,k} a_n g_{n,k} e^{-it(\lambda_n +m)}(\lambda_n +m) e_{n,k}
$$
and since $\lambda_n$ are the eigenvalues of $\lap_{\S^d}$ and $m = \E(|Y(t,x)|^2)$, 
$$
i\partial_t Y = (-\lap_{\S^d} +m) Y = (-\lap_{\S^d} +\E(|Y(x)|^2)) Y .
$$
Therefore, $Y$ solves \eqref{eqonrv}.

The fact that the law of $Y$ does not depend on $t$ is due to the invariance of the law of a Gaussian under rotations.

Finally, we have 
$$
\|Y\|_{L^2(\Omega, H^1(\S^d))}^2 = \sum_{n,k} |a_n|^2 \lambda_n  = \sum_n |a_n|^2 \lambda_n N_n 
$$
which converges since $ |a_n|^2 \lambda_n N_n \sim  n^{d+1} |a_n|^2 $ up to a constant.
\end{proof}

\begin{remark} The random variable $Y$ corresponds to $\gamma = f(-\lap_{\S^d})$ with $|a_n|^2 = f(\lambda_n)$. Indeed, for all $v\in L^2(\S^d)$,
$$
\E(\an{Y(t),v}Y(t)) = \sum_{n,k} |a_n|^2 \an{e_{n,k},v} e_{n,k} = f(-\lap)(v).
$$
Note that the operator does not depend on $t$, which makes $f(-\lap_{\S^d})$ a stationary state for \eqref{eqonop}.
\end{remark}

\paragraph{On $\T^d$}

Let $(a_k)_{k\in \Z^d}$ be a sequence of complex numbers such that
$$
\sum_{k\in \Z^d} (1+ |k|^2) |a_k|^2 < \infty
$$
where $|k|^2 = \sum_i k_i^2$ and let $(g_k)$ be a sequence of independent centred normalised and complex Gaussian variables. We set
$$
Y_0(x) = \sum_{k\in \Z^d} a_k g_k e^{ikx}
$$
with $kx = \sum_i k_i x_i$. Let $m = \sum_k |a_k|^2$, and 
$$
Y(t) = \sum_{k\in\Z^d} a_k e^{-it(k^2 + m)}g_k e^{ikx}.
$$

\begin{proposition} The random variable $Y$ is a solution to \eqref{eqonrv} belonging to $L^2(\Omega,H^1(\T^d))$ whose law does not depend on $t$.\end{proposition}

\begin{proof} We have that thanks to the independence of the $g_k$ that
$$
\E (|Y(t,x)|^2) = m
$$
and 
$$
i\partial_t Y =  (-\lap + m) Y.
$$
The law of $Y$ does not depend on $t$ as the law of Gaussian variable is invariant under rotations. The variable $Y$ belongs to $L^2(\Omega, H^1(\S^d))$ since 
$$
\sum_{k\in \Z^d} (1+ |k|^2) |a_k|^2 < \infty
$$
\end{proof}

\begin{remark} The random variable $Y$ corresponds to the Fourier multiplier $\gamma$ by $|a_k|^2$, that is for all $t$, $\gamma_{Y(t)} = \gamma$, which makes $\gamma$ a stationary state of \eqref{eqonop}.
\end{remark}

\paragraph{On $\R^d$}

Let $W$ be a $d$-dimensional complex centred random Gaussian process such that for all \\ $k=(k_1,\hdots, k_d) \in \R^d$ and $k'=(k'_1,\hdots, k'_d) \in \R^d$, we have
$$
\E(W(k)\overline{W(k')}) = \left \lbrace{\begin{array}{ll}
0 & \textrm{ if there exists } j\in [|1,d|] \textrm{ such that } k_jk_j' <0 \\
\prod_{j=1}^d \min(|k_j|,|k_j'|) & \textrm{ otherwise.}
\end{array}} \right.
$$
In other terms, $\E(dW(k) \overline{dW(k')} = dk dk' \delta(k-k')$ and $W(0) = 0$ where $\delta$ is the Dirac delta in dimension $d$.

For more information about Gaussian processes, we refer to \cite{simonPphi2}.

Let $f\in L^2(\R^d)$ such that $k\mapsto \sqrt{1+|k|^2} f(k)$ belongs to $L^2(\R^d)$ and set $Y_0$ the random variable
$$
Y_0 (x) = \int f(k) e^{inx} dW(k) .
$$
where $k= (k_1,\hdots ,k_d)$ and $kx = \sum_i k_i x_i$. We write $m = \int_{\R^d} |f(k)|^2 dk$ and 
$$
Y(t,x)= \int e^{-i(k^2+m)t} f(k)  e^{ikx}  dW(k) .
$$

\begin{proposition} The random variable $Y(t,x)$ is a solution of \eqref{eqonrv} whose law does not depend on $t$.\end{proposition}

\begin{proof} The random variable $Y$ satisfies 
$$
i\partial_t Y = -\lap Y + m Y.
$$
We have 
\begin{multline*}
\E(|Y(t,x)|^2)= \E \Big( \Big|\int e^{-i(k^2+m)t} f(k)  e^{ikx} dW(k)\Big|^2 \Big) \\
= \int |e^{-i(k^2+m)t} f(k)  e^{ikx}|^2 dk = \int |f(k)|^2 dk = m.
\end{multline*}
Hence $Y$ satisfies \eqref{eqonrv}. The law of $Y$ does not depend on time because Gaussian variables are invariant under rotations.
\end{proof}

\begin{remark} The random variable $Y$ is a natural invariant in the sense of the law for the equation \eqref{eqonrv}. Nevertheless, it is not in $L^2(\R^d)$, or in $H^1(\R^d)$ but in $L^2_{\textrm{loc}}(\R^d)$, and in $H^1_{\textrm{loc}}(\R^d)$. That means that one cannot use the usual arguments of continuity in he initial datum or scattering to prove the stability of $Y$. We comment this lack of localisation in Section \ref{sec-scatt}.\end{remark}

\begin{remark} The operator associated to $Y$ is the Fourier multiplier by $|f(k)|^2$. Indeed,
$$
\E(\an{Y,v}Y) = \E\Big( \int dy v(y) \int \overline{dW(k)} \overline{f(k)} e^{-iky} \int dW(k') f(k') e^{ik'x} \Big)
$$
which yields, since $\E(dW(k) \overline{dW(k')}) = dk \delta(k-k')$,
$$
\E(\an{Y,v}Y) = \int dk e^{ikx} |f(k)|^2 \hat v(k).
$$
\end{remark}

We sum up the parallels we have made in this section in the following table.

\bigskip

\begin{tabular}{|c|c|c|} \hline
 & Operator level & Random variable level \\\hline
Equation & $i\partial_t \gamma = [-\lap + \rho_\gamma, \gamma]$ & $i\partial_t X = -\lap X + \E(|X|^2) X$\\\hline
Solution & $\gamma = \E(|X\times X|)$ & $X$ \\\hline
Initial datum & $\gamma_0$ & Gaussian field with covariance $\gamma_0$ \\\hline
Compact initial datum & $\gamma_0 = \sum_n |\alpha_n|^2 |u_n\times u_n|$ & $X_0 = \sum_n \alpha_n u_n g_n $ with $(g_n)_n$ i.i.d $\mathcal N(0,1)$\\\hline
Possible condition & $\textrm{Tr}(\gamma_0 (1-\lap))<\infty$  & $X_0 \in L^2(\Omega, H^1)$ \\
on the initial datum  & & \\\hline
Equilibrium on $\S^d$ & $\sum_{n,k} |a_n|^2 |e_{n,k}\times e_{n,k}|$ & $\sum_{n,k} a_n e^{-i(\lambda_n +m)t}e_{n,k}g_{n,k} $ \\\hline
Equilibrium on$ \T^d$ & $\widehat{\gamma_0 \varphi}(k) = |f(k)|^2 \hat \varphi (k)$ & $\sum_{k\in \Z^d} f(k) g_k \frac{e^{ikx}}{\sqrt{2\pi}}e^{-i(k^2 +m)t}$ \\\hline
Equilibrium on $\R^d$ & $\widehat{\gamma_0 \varphi}(k) = |f(k)|^2 \hat \varphi (k)$ & $\int_{\R^d} f(k) e^{ikx} e^{-i(k^2+m)t}dW_k$ \\\hline \end{tabular}

\bigskip

Finally, we make one last remark, which is also the main subject of Section \ref{sec-scatt}. In Section \ref{sec-scatt}, we prove that \eqref{eqonrv} scatters when the initial datum is in $H^1(\R^3)$. This may explain why the equilibria are not localised. At least, it explains why they are not in $H^1(\R^3)$. Indeed, if $Y(t)$ is both an equilibrium and in $H^1(\R^3)$. Then, as it scatters it converges to the solution to the linear equation 
$$
i\partial_t X  = -\lap X 
$$
with an initial datum in $H^1(\R^3)$ and because of dispersion in $\R^3$, $Y(t)$ goes to $0$ in some sense as $t$ goes to $\infty$. But it is impossible, unless $Y(t)$ is almost surely $0$, as the law of $Y(t)$ does not depend on $t$. We discuss this in more detail in Section \ref{sec-scatt}.

\subsection{Main results}

Throughout the paper, we give some results derived from the Schr\"odinger equation's theory for the equation on $X$, that is \eqref{eqonrv}, such as global well-posedness in the energy space in $\T^d,\R^d,\S^d$ for $d=2,3$ or scattering in $L^2(\Omega,H^1(\R^3))$ and in the case of the focusing equation, existence of blow-up solutions, but we choose to state here two results of global well-posedness for the equation on $\gamma$, that is \eqref{eqonop}.

\begin{theo}\label{th-th1} Let $M \in \{\S^2,\S^3,\T^2,\T^3\}$. Let $\Sigma$ be the set of non-negative operators $\gamma$ on $M$ such that $\textrm{Tr }((1-\lap)\gamma)<\infty$. Let $d$ be the distance on $\Sigma$ defined in Definition \ref{def-dist}. 

The equation \eqref{eqonop} is well-posed in $\mathcal C(\R,\Sigma)$ in the sense that for all $\gamma_0 \in \Sigma$ there exists a solution of \eqref{eqonop} with initial datum $\gamma_0$ in $\mathcal C(\R,\Sigma)$, this solution is unique in $\mathcal C(\R,\Sigma)$ and the flow thus defined is continuous in the initial datum.
\end{theo}

\begin{theo}\label{th-th2} Let $M\in \R^2,\R^3$. Let $f$ be a bounded map on $M$ such that $\an k |f(k)| \in L^2$. Let $\gamma_f$ be the Fourier multiplier by $|f|^2$. Let $\gamma_f^{1/2}$ be the Fourier multiplier by $f$. Let $\Sigma_f$ be the set of non-negative operators $\gamma$ on $M$ such that there exists a square root of $\gamma$, $\gamma^{1/2}$ such that $Q = \gamma^{1/2} - \gamma_f^{1/2}$ satisfies $\textrm{Tr }(Q^*(1-\lap)Q) <\infty$. The set $\Sigma_f$ with the distance $d$ is a well-defined metric space. 

The equation \eqref{eqonop} is well-posed in $\mathcal C(\R,\Sigma_f)$ in the sense that for all $\gamma_0 \in \Sigma$ there exists a solution of \eqref{eqonop} with initial datum $\gamma_0$ in $\mathcal C(\R,\Sigma_f)$, this solution is unique in $\mathcal C(\R,\Sigma)$ and the flow thus defined is continuous in the initial datum.
\end{theo}

\section{Well-posedness on  \texorpdfstring{$\mathbb S^d$}{Sd} and \texorpdfstring{$\mathbb T^d$}{Td}} \label{sec-WPST}

The goal of this section is to prove global well-posedness results in the energy space.

\subsection{Local well-posedness on \texorpdfstring{$\mathbb S^2$}{S2} and \texorpdfstring{$\mathbb T^2$}{T2}}

In this subsection, we explain why the equation \eqref{eqonrv} is locally well-posed in $L^2(\Omega,H^1(M_2))$ with $M_2 = \S^2$ or $\T^2$. The proof is very similar to the deterministic case and we do not claim any novelty regarding these techniques. We include the proof to explain how to deal with the probability part. This analysis could be applied to more general manifolds of dimension $2$, as the main tool, that is Strichartz estimates, holds in a more general setting than $\S^2$ or $\T^2$. We refer to \cite{burgertzvStri}.

We recall that from \cite{burgertzvStri}, Strichartz estimates on the sphere implies a loss of derivative. We use the following Strichartz estimate : for all $f\in H^1(M_2)$ and with $S(t) = e^{it\lap}$,
\begin{equation}\label{str}
\|S(t)f\|_{L^3([-1,1],L^\infty(M_2))} \leq C \|S(t)f\|_{L^3([-1,1],W^{1/2,6}(M_2))} \leq C \|f\|_{H^1(M_2)}.
\end{equation}

Let $T\leq 1$. We call $\L_T$ the space $ \mathcal C([-T,T], H^1(M_2)) \cap L^3([-T,T],L^\infty(M_2))$ normed by
$$
\|f\|_{\L_T} = \|f\|_{L^\infty([-T,T],H^1(M_2))}+ \|f\|_{L^3([-T,T],L^\infty(M_2))}.
$$

\begin{proposition}\label{prop-lwpS} Let $R\geq 0$. There exists $C$ such that with $T = \frac1{CR^6}$, and for all $X_0$ such that $\|X_0\|_{L^2(\Omega, H^1(M_2))} \leq R$ the equation \eqref{eqonrv} with initial datum $X_0$ has a unique solution $X$ in $L^2(\Omega,\L_T)$, this solution is continuous in the initial datum and satisfies
$$
\|X\|_{L^2(\Omega,\L_T)} \leq CR.
$$
\end{proposition}

\begin{proof} We proceed with a contraction argument.

The Duhamel formulation of \eqref{eqonrv} is 
$$
X(t,x) = S(t) X_0 -i \int_{0}^t S(t-\tau) \Big(\E(|X(\tau,x)|^2)) X(\tau,x)\Big) d\tau.
$$
Let 
$$
A(X)(t,x) = S(t) X_0 -i \int_{0}^t S(t-\tau) \Big(\E(|X(\tau,x)|^2)) X(\tau,x)\Big) d\tau.
$$
We prove that $A$ is contracting in a ball of radius $CR$.

Thanks to Strichartz estimate \eqref{str} and the invariance of the $H^1$ norm under $S(t)$, we get
$$
\|A(X)\|_{\L_T} \leq C_1 \|X_0\|_{H^1(M_2)} + C_1\int_{-T}^T \|\E(|X(\tau)|^2)) X(\tau)\|_{H^1(M_2)}.
$$
We estimate the $H^1$ norm by  $\|\cdot\|_{L^2(M_2)} + \|\grad \cdot\|_{L^2(M_2)}$. We have 
$$
\|\E(|X(\tau)|^2)) X(\tau)\|_{L^2(M_2)} \leq \|\E(|X(\tau)|^2))\|_{L^\infty} \|X(\tau)\|_{L^2}.
$$
As a consequence of Minkowski inequality, we have $\|\cdot\|_{L^p_*,L^q_z} \leq \|\cdot\|_{L^q_z,L^p_*}$ as long as $p\geq q$ hence
$$
\|\E(|X(\tau)|^2))\|_{L^\infty} \leq \E(\|X(\tau)\|_{L^\infty}^2)
$$
Integrating in time yields
$$
\int_{-T}^T \|\E(|X(\tau)|^2)) X(\tau)\|_{L^2(M_2)} \leq \E\Big(\int_{-T}^T \|X(\tau)\|_{L^\infty}^2 \Big) \|X(\tau)\|_{L^\infty([-T,T],L^2(M_2))}
$$
and using H\"older inequality in the integral in time gives
\begin{equation} \label{estim1}
\int_{-T}^T \|\E(|X(\tau)|^2)) X(\tau)\|_{L^2(M_2)} \leq T^{1/3} \E \Big( \|X\|_{\L_T}^2\Big) \|X\|_{\L_T}.
\end{equation}

For the term including derivatives, we have
$$
\grad \Big( \E(|X(\tau,x)|^2)) X(\tau,x)\Big) = 2\Re \Big( \E \Big( \overline X(\tau,x) \grad X(\tau,x)\Big)\Big) X(\tau,x) + \E(|X(\tau,x)|^2) \grad X(\tau,x).
$$
Hence, thanks to H\"older inequality on the mean value, we get
$$
\Big|\grad \Big( \E(|X|^2)) X\Big) \Big| \leq 2 \|\grad X(\tau,x)\|_{L^2(\Omega) } \|X\|_{L^2(\Omega) }|X|+ \E(|X|^2) |\grad X|.
$$
We take the $L^2$ norm in space, we get
$$
\|\grad \Big( \E(|X|^2)) X\Big)\|_{L^2(M_2)} \leq 2 \|\grad X(\tau,x)\|_{L^2(M_2,L^2(\Omega)) } \|X\|_{L^\infty(M_2, L^2(\Omega)) }\|X\|_{L^\infty(M_2)}+ \E(\|X\|_{L^\infty(M_2)}^2)) \| X\|_{H^1}.
$$
Integrating in time yields
\begin{multline*}
\int_{-T}^T \|\grad \Big( \E(|X|^2)) X\Big)\|_{L^2(M_2)} d\tau \leq \\
T^{1/3} \Big( 2 \|\grad X(\tau,x)\|_{L^\infty([-T,T],L^2(M_2,L^2(\Omega))) } \|X\|_{L^3([-T,T],L^\infty(M_2, L^2(\Omega))) }\|X\|_{L^3([-T,T],L^\infty(M_2))}\\
+ \|X\|_{L^3([-T,T],L^2(\Omega,L^\infty(M_2)))}^2 \| X\|_{L^\infty([-T,T],H^1)}\Big).
\end{multline*}

We recall that $\|\cdot\|_{L^p_*,L^q_z} \leq \|\cdot\|_{L^q_z,L^p_*}$ as long as $p\geq q$ hence
\begin{eqnarray*}
\|\grad X(\tau,x)\|_{L^\infty([-T,T],L^2(M_2,L^2(\Omega))) } &\leq &\|\grad X(\tau,x)\|_{L^2(\Omega,L^\infty([-T,T],L^2(M_2))) }\leq \|X\|_{L^2(\Omega,\L_T)}\\
\|X\|_{L^3([-T,T],L^\infty(M_2, L^2(\Omega))) } & \leq &\|X\|_{L^2(\Omega, L^3([-T,T],L^\infty(M_2))) }\leq \|X\|_{L^2(\Omega,\L_T)}\\
\|X\|_{L^3([-T,T],L^2(\Omega,L^\infty(M_2)))} & \leq & \|X\|_{L^2(\Omega,L^3([-T,T],L^\infty(M_2)))}\leq \|X\|_{L^2(\Omega,\L_T)}.
\end{eqnarray*}
We get
\begin{equation}\label{estim2}
\int_{-T,T} \|\grad\Big(\E(|X(\tau)|^2)) X(\tau)\Big)\|_{L^2(M_2)} \leq 3T^{1/3} \E \Big( \|X\|_{\L_T}^2\Big) \|X\|_{\L_T}.
\end{equation}

Putting together \eqref{estim1} and \eqref{estim2} yields
$$
\|A(X)\|_{\L_T} \leq C_1\|X_0\|_{H^1(M_2)} + 4C_1 T^{1/3} \E \Big( \|X\|_{\L_T}^2\Big) \|X\|_{\L_T}.
$$
Taking its $L^2$ norm in probability yields
$$
\|A(X)\|_{L^2(\Omega,\L_T)} \leq C_1R + 4C_1 T^{1/3} \|X\|_{L^2(\Omega,\L_T)}^3.
$$
Hence, for $T \leq \frac1{(4C_1)^3 (2C_1 R)^6}$, the ball of $L^2(\Omega, \L_T)$ of radius $2C_1 R$ is stable under $A$. 

We prove that $A$ is contracting on this ball, we have 
$$
A(X_1) - A(X_2) = \int_{0}^t S(t-\tau) \Big( \E(|X_1|^2) X_1 - \E(|X_2|^2)X_2\Big) d\tau.
$$
Since 
$$
\E(|X_1|^2) X_1 - \E(|X_2|^2)X_2 = \E(|X_1|^2) (X_1-X_2) + \E(X_1\overline{(X_1-X_2)}X_2 + \E( \overline{X_2} (X_1-X_2))X_2,
$$
buying doing the same computations as previously, we get
$$
\|A(X_1) - A(X_2)\|_{L^2(\Omega, \L_T)} \leq 4C_1 T^{1/3}\Big( \|X_1\|_{L^2(\Omega, \L_T)}^2 +  
\|X_2\|_{L^2(\Omega, \L_T)}^2\Big)  \|X_1 - X_2\|_{L^2(\Omega, \L_T)}
$$
thus on the ball of radius $2C_1 R$ we get
$$
\|A(X_1) - A(X_2)\|_{L^2(\Omega, \L_T)} \leq 12C_1 (2C_1R)^2 T^{1/3} \|X_1 - X_2\|_{L^2(\Omega, \L_T)}
$$
and for $T < \frac1{(12C_1)^3 (2C_1 R)^6 }$ the map $A$ is contracting, which concludes the proof.
\end{proof}

\subsection{Local well-posedness on \texorpdfstring{$\mathbb S^3$}{S3} and \texorpdfstring{$\mathbb T^3$}{T3}}

In this subsection, we prove local well-posedness of \eqref{eqonrv} on $\T^3$ and $\S^3$, relying on \cite{bourgainlattice} and \cite{BGTmultilin}. Once again, we adapt the techniques from these papers to deal with the probability space but we do not claim any novelty regarding the deterministic analysis. We remark that as opposed to dimension $2$, one cannot use the same techniques for more general manifolds.

Let $M_3 = \T^3$ or $\S^3$. Let $(e_k)_k$ be an orthonormal basis of $L^2(M_3)$ consisting in eigenvalues of the Laplace-Beltrami operator associated to the eigenvalue $(\lambda_k)_k$. For all $u\in L^2(M_3)$, let $u_k = \an{e_k,u}$.

Let $X^{s,b}(M_3)$ be the Bourgain space induced by the norm
\begin{equation}\label{bourgainspace}
\|u\|_{X^{s,b}(M_3)}^2 = \sum_k \an{\lambda_k}^s\|\an{\lambda_k + \tau}\hat u_k(\tau)\|_{L^2(\R_\tau)}^2
\end{equation}
where $\hat u_k$ is the Fourier transform in time of $u_k$.

Finally, for $T\leq 1$, let $X_T^{s,b}(M_3)$ be the Bourgain space induced by the norm
\begin{equation}\label{defBT}
\|u\|_{X_T^{s,b}(M_3)} = \inf \{\|w\|_{X^{s,b}(M_3)} \,|\, w_{|[-T,T]} = u \}
\end{equation}

Adapting the proof of Proposition 2.11 in \cite{BGTbilin} to $M_3$ as it is done in \cite{BGTmultilin}, one gets the following estimate
\begin{equation}\label{linest}
\big\| S(t) u_0 + \int_{0}^t S(t-\tau) F(x,\tau) d\tau \big\|_{X_T^{1,b}(M_3)}  \lesssim \|u_0\|_{H^1(M_3)} + T^{1-b-b'}\|F\|_{X^{1,-b'}(M_3)} 
\end{equation}
for $(b,b')$ satisfying $0<b'<\frac12 < b$ and $b+b'<1$.

From \cite{bourgainlattice} for the torus and from \cite{BGTmultilin} for the sphere, one can deduce the following trilinear estimate : there exists $(b,b') \in \R^2$ such that $0<b'<\frac12<b$ and $b+b'<1$ such that
\begin{equation}\label{trilinest}
\|uvw\|_{X^{1,-b'}(M_3)} \lesssim \|u\|_{X^{1,b}(M_3)}\|v\|_{X^{1,b}(M_3)}\|w\|_{X^{1,b}(M_3)} .
\end{equation}
We remark that the constant implied by \eqref{trilinest} and the couple $(b,b')$ may depend on $M_3$.

\begin{proposition}\label{prop-lwpM3} Let $R\geq 0$. There exists $C$ and  $T = T(R)$, such that for all $X_0$ such that $\|X_0\|_{L^2(\Omega, H^1(M_3))} \leq R$ the equation \eqref{eqonrv} with initial datum $X_0$ has a unique solution $X$ in \\$L^2(\Omega,X_T^{1,b}(M_3))$, this solution is continuous in the initial datum and satisfies
$$
\|X\|_{L^2(\Omega,X_T^{1,b}(M_3))} \leq CR.
$$
\end{proposition}

\begin{remark}\label{rem-persistencehireg} First, the $X^{s,b}$ norm controls the $H^s$ norm. We also have persistence of higher regularity in the sense that if the initial datum belongs to $H^s(M_3)$ with $s>1$, then the solution remains in $X^{s,b}_T(M_3)$ for a time $T$ which depends only on the $H^1(M_3)$ norm of the initial datum. This is due to the fact that 
$$
\big\| \int_{0}^t S(t-\tau) \Big(|u(\tau)|^2 u(\tau)\Big) d\tau\big\|_{X^{s,b}_T(M_3)} \lesssim T^\alpha \|u\|_{X^{s,b}_T}\|u\|_{X^{1,b}_T}^2
$$
for some positive $\alpha$ (see \cite{bourgainlattice, BGTmultilin}) which eventually leads to
$$
\big\| \int_{0}^t S(t-\tau) \Big(\E(|X(\tau)|^2) X(\tau)\Big) d\tau\big\|_{L^2(\Omega,X^{s,b}_T(M_3))} \lesssim T^\alpha \|X\|_{L^2(\Omega,X^{s,b}_T(M_3))}\|X\|_{L^2(\Omega,X^{1,b}_T(M_3))}^2.
$$    \end{remark}

Before we prove this proposition, we prove the following lemma.

\begin{lemma}\label{lem-trilin} For all $u,v,w \in L^2(\Omega, X^{1,b}(M_3))$, we have 
$$
\|\E(uv)w\|_{L^2(\Omega, X^{1,-b'}(M_3))} \lesssim \|u\|_{L^2(\Omega, X^{1,b}(M_3))}\|v\|_{L^2(\Omega, X^{1,b}(M_3))}\|w\|_{L^2(\Omega, X^{1,b}(M_3))}.
$$
\end{lemma}

\begin{proof} We proceed by duality. Let $h$ in the dual of $L^2(\Omega, X^{1,-b'}(M_3))$ that is $L^2(\Omega, X^{-1,b'})$. We have 
$$
\an{\E(uv)w ,h}_{\Omega \times M_3} = \int_{\Omega }\int_{M_3} \E(u(x)v(x))w(\omega_1,x) h(\omega_1,x) dx dP(\omega_1)
$$
where $\an{\cdot, \cdot}_{\Omega \times M_3}$ is the inner product in $\Omega \times M_3$. We replace $\E$ by an integral over $\Omega$ we get
$$
\an{\E(uv)w ,h}_{\Omega \times M_3} = \int_{\Omega \times \Omega } \int_{M_3} u(\omega_2,x)v(\omega_2,x) w(\omega_1,x) h(\omega_1,x)dx dP(\omega_1)dP(\omega_2).
$$
Using \eqref{trilinest} for $u(\omega_2)$, $v(\omega_2)$ and $w(\omega_1)$, we get
$$
|\an{\E(uv)w ,h}_{\Omega \times M_3}| \lesssim \int_{\Omega \times \Omega } \| u(\omega_2)\|_{X^{1,b}(M_3)}\|v(\omega_2) \|_{X^{1,b}(M_3)}\|w(\omega_1)\|_{X^{1,b}(M_3)}\| h(\omega_1)\|_{X^{-1,b'}}.
$$
We can use Cauchy-Schwartz inequality on $\omega_1$ and on $\omega_2$ to get
$$
|\an{\E(uv)w ,h}_{\Omega \times M_3}| \lesssim  \| u\|_{L^2(\Omega,X^{1,b}(M_3))}\|v \|_{L^2(\Omega,X^{1,b}(M_3))}\|w\|_{L^2(\Omega,X^{1,b}(M_3))}\| h\|_{L^2(\Omega,X^{-1,b'})}.
$$
which concludes the proof. \end{proof}

\begin{proof}[Proof of Proposition \ref{prop-lwpM3}.] Let 
$$
A(X) =  S(t) X_0 + \int_{0}^t S(t-\tau) \E(|X|^2)X d\tau
$$
such that the Duhamel formulation of \eqref{eqonrv} is 
$$
X = A(X).
$$
We proceed with a contraction argument on $L^2(\Omega, X_T^{1,b}(M_3))$. We have, thanks to \eqref{linest},
$$
\|A(X)\|_{L^2(\Omega, X_T^{1,b}(M_3))} \lesssim \|X_0\|_{L^2(\Omega,H^1(M_3))} + T^{1-(b+b')} \|\E(|X|^2X)\|_{L^2(\Omega,X^{1,-b'}(M_3))}
$$
and thanks to Lemma \ref{lem-trilin},
$$
\|\E(|X|^2X)\|_{L^2(\Omega,X^{1,-b'}(M_3))} \lesssim \|X\|_{L^2(\Omega, X^{1,b}(M_3)_T)}^3.
$$
For the same reasons
$$
\|A(X_1) - A(X_2)\|_{L^2(\Omega, X_T^{1,b}(M_3))} \lesssim  T^{1-b-b'}\|X_1-X_2\|_{L^2(\Omega, X_T^{1,b}(M_3))}(\|X_1\|_{L^2(\Omega, X_T^{1,b}(M_3))}\|X_2\|_{L^2(\Omega, X_T^{1,b}(M_3))})\; .
$$
Hence, as $b+b'<1$, there exist $C$ and $T(R)$ such that the ball of $L^2(\Omega, X_T^{1,b}(M_3))$ of radius $CR$ is stable under $A$ and such that $A$ is contracting on this ball, which concludes the proof.
\end{proof}

\subsection{Global Well-posedness}

Let $M = M_2$ or $M_3$. In this subsection, we prove global well-posedness in $H^1(M)$ using energy methods. 

\begin{lemma} Let
$$
\mE (X) = \mE_{\textrm{kin}}(X) + \mE_{\textrm{pot}}(X) = \frac12 \int_{\Omega\times M} \overline X (1-\lap)X + \frac14 \int_{M} \E(|X|^2)^2.
$$
The quantity $\mE$ is invariant under the flow of \eqref{eqonrv}.
\end{lemma}

\begin{proof} Thanks to an approximation argument and the persistence of higher regularity, see Remark \ref{rem-persistencehireg}, we can assume that $X$ is regular enough so that the computations below are justified.

Let $X(t)$ be a solution of \eqref{eqonrv} and let us differentiate $\mE(X(t))$. We have
$$
\partial_t \mE_{\textrm{kin}} (X(t)) = \Re \Big( \int_{\Omega \times M} (\partial_t \overline X ) X \Big) + \Re \Big( \int_{\Omega \times M} (\partial_t \overline X ) (-\lap X) \Big).
$$
Because $X$ satisfies \eqref{eqonrv}, we have 
$$
\Re \Big( \int_{\Omega \times M} (\partial_t \overline X ) X \Big) = \Im \Big( \int_{\Omega \times M} (\overline{-\lap X + \E(|X|^2)X})X\Big) 
$$
and because of the imaginary part, this quantity is zero. Therefore,
$$
\partial_t \mE_{\textrm{kin}} (X(t)) = \Im \Big( \int_{\Omega \times M} \overline{ i\partial_t X}  (-\lap X) \Big).
$$

Differentiating $\mE_{\textrm{pot}}(X(t))$ yields
$$
\partial_t \mE_{\textrm{pot}}(X(t)) = \frac12 \int_{M} \E(|X|^2) \partial_t (\E(|X|^2)).
$$
As $\partial_t$ and $\E$ commute, we get
$$
\partial_t \mE_{\textrm{pot}}(X(t)) = \int_{M} \E(|X|^2) \E(\Re (\overline{\partial_t X}X))
$$
and we write the second expectation as an integral in the sense that
$$
\partial_t \mE_{\textrm{pot}}(X(t)) = \int_{\Omega \times M} \E(|X|^2) \Re (\overline{\partial_t X}X)
$$
which finally yields
$$
\partial_t \mE_{\textrm{pot}}(X(t)) = \Re \Big( \int_{\Omega \times M} \overline{\partial_t X}\E(|X|^2) X\Big) = \Im \Big( \int_{\Omega \times M} \overline{i\partial_t X}\E(|X|^2) X\Big).
$$
Summing the derivatives of $\mE_{\textrm{kin}}(X(t))$ and $\mE_{\textrm{pot}}(X(t))$ gives
$$
\partial_t \mE(X(t)) =  \Im \Big( \int_{\Omega \times M} \overline{i\partial_t X} ( -\lap X + \E(|X|^2) X )
$$
and because of the imaginary part and the fact that $X$ satisfies \eqref{eqonrv}, we get $\partial_t \mE(X(t)) =0$, which concludes the proof. \end{proof}

Since $\mE$ controls the $L^2(\Omega,H^1(M))$ norm, we get the following proposition.

\begin{proposition}\label{prop-gwpS} The equation \eqref{eqonrv} is globally well-posed in $L^2(\Omega, H^1(M))$. \end{proposition}

\subsection{Continuity with regard to the initial datum, interpretation in terms of law}\label{subsec-interplaw}

\begin{remark} The first thing one can remark is that we have continuity in the initial datum. Indeed, let $X_1$ and $X_2$ be the solutions of \eqref{eqonrv} with initial data $X_1(0)$ close to $X_2(0)$. Let $R$ be the maximum of $\|X_1(0)\|_{L^2(\Omega, H^1(M))}$ and $\|X_2(0)\|_{L^2(\Omega, H^1(M))}$. Then, up to times of order $R^{-6}$ for $M=M_2$ or $R^{-N}$ for some $N$ for $M_3$, we have 
$$
\|X_1(t) - X_2(t)\|_{L^2(\Omega, H^1(M))} \leq C \|X_1(0) - X_2(0)\|_{L^2(\Omega, H^1(M))}
$$
with $C$ independent from $R$.

Iterating this estimate for longer times, in view of the conservation of the energy $\mE$ yields estimates such as 
$$
\|X_1(t) - X_2(t)\|_{L^2(\Omega, H^1(M))} \leq Ce^{cR^N(1+|t|)} \|X_1(0) - X_2(0)\|_{L^2(\Omega, H^1(M))}.
$$
We note that the spaces we used in the local well-posedness were not optimal at least for $M_2$, and one could probably reach finite time estimates for times of order $R^{-(4+\varepsilon)}$ for $M_2$, $\varepsilon > 0$.
\end{remark}

\begin{remark}\label{rem-wass}Let $\rho_1^t$  for all $t\in \R$ be the law of, or the measure induced by, $X_1(t)$ and $\rho_2^t$ be the law of $X_2(t)$. Let $d_2$ be the Wasserstein distance of order $2$ on the measures on $H^1(M)$, that is
$$
d_2(\mu_1,\mu_2) = \inf_{\mu \in \textrm{Marg}(\mu_1,\mu_2)} \Big(\int \|u-v\|_{H^1}^2 d\mu(u,v) \Big)^{1/2}
$$
where $\mu_i$ for $i=1,2$ are measures on $H^1(M)$ such that $\int \|u\|^2_{H^1(M)} d\mu_i(u)$ are finite and $\textrm{Marg}(\mu_1,\mu_2)$ is the set of measures on $H^1(M)\times H^1(M)$ whose marginals are $\mu_1$ and $\mu_2$. Let $\mu^t$ be the law of $(X_1(t),X_2(t))$. Since $\mu^t$ has for marginals the law of $X_1(t)$, that is $\rho_1^t$ and the law of $X_2(t)$, that is $\rho_2^t$, we get that
$$
d_2(\rho_1^t, \rho_2^t) \leq \Big(\int \|u-v\|_{H^1}^2 d\mu^t(u,v) \Big)^{1/2} = \|X_1(t) - X_2(t)\|_{L^2(\Omega, H^1(M))}
$$
and therefore
$$
d_2(\rho_1^t, \rho_2^t) \leq  Ce^{cR^N(1+|t|)} \|X_1(0) - X_2(0)\|_{L^2(\Omega, H^1(M))}.
$$
Since this is true for all $X_1(0)$ and $X_2(0)$ with laws $\rho_1^0$ and $\rho_2^0$, we get
$$
d_2(\rho_1^t, \rho_2^t) \leq  Ce^{cR^N(1+|t|)} d_2(\rho_1^0,\rho_2^t)
$$
which gives a continuity in the law of the initial datum.

If we replace the $X_2(t)$ by $Y(t)$, as the law of $Y(t)$, called $\nu$, does not depend on $t$, we get that
$$
d_2(\rho_1^t,\nu)  \leq  Ce^{cR^N(1+|t|)}d_2(\rho_1^0,\nu) 
$$
which is a result of stability for $\nu$ under the flow of the equation \eqref{eqonrv}.
\end{remark}

\begin{remark} We have what we could call orbital stability in the sense that as $\mE$ is conserved, $\mE(X(t)) - \mE(Y(t))$ does not depend on time. Nevertheless, $\mE(X(t)) - \mE(Y(t))$ does not control a norm of $X-Y$. \end{remark}

\section{Well-posedness on the Euclidean space}\label{sec-WPE}

In $\R^d$, the equilibria $Y$ are not localised. In particular, the law of $Y$ is invariant under translations. In this section, we prove the existence of dynamics for perturbations around $Y$ which are localised, in the sense that we prove global well-posedness for solutions $X$ of \eqref{eqonrv} that are written $X = Y + Z$ where $Z$ is localised as $Z \in L^2(\Omega, H^1(\R^d))$.

\subsection{Perturbed equation and local well-posedness for \texorpdfstring{$d\leq 3$}{d<=3}}

We perturb $Y$. Let $X = Y + Z$ such that $Z(0) = Z_0$ is in $L^2(\Omega , H^1(\R^d))$. The random variable $Z$ solves the equation
\begin{equation}\label{eqonZ}
i\partial_t Z = (-\lap + m)Z + \Big( \E(|Z|^2) + 2\Re ( \E(\overline Y Z)) \Big) (Y+Z).
\end{equation}

Let $\mathcal L_T = L^p([-T,T],L^\infty(\R^d)) \cap \mathcal C([-T,T], H^1(\R^d))$ with $p = 4 \frac{d+1}{d(d-1)}$. We prove local well-posedness in $L^2(\Omega,\mathcal L_T)$. First, in dimension $d\leq 3$, we have Strichartz estimates in the sense that there exists $C$ such that for all $g\in H^1$,
\begin{equation}\label{strR}
\|S(t) f\|_{\mathcal L_T} \leq \|f\|_{H^1}.
\end{equation}

Indeed, for $d\leq 3$, $p>2$ and with $q$ such that
$$
\frac2{p} + \frac{d}{q} = \frac{d}2
$$
that is $q = d+1$ or $\frac1{q}< \frac1{d}$, we get that thanks to Sobolev embeddings
$$
\|S(t)g\|_{L^p,L^\infty} \leq C\|S(t) g\|_{L^p,W^{q,1}}
$$
and thanks to Strichartz estimates and the commutation of the differential operator $D=\sqrt{1-\lap}$ and $S(t)$
$$
\|S(t)g\|_{L^p,L^\infty} \leq C \|g\|_{H^1}.
$$

Let
$$
m_2 = \int k^2 |f(k)|^2 dk < \infty .
$$
That means that for all $t$ and all $x$, $\E(|\grad Y(t,x)|^2) = m_2 < \infty$.

\begin{proposition}\label{prop-lwp}There exists $C$ such that for all $Z_0 \in L^2(\Omega, H^1)$, with 
$$
T =  \min \Big(1,\frac1{C(m+m_2)},\frac1{C(\sqrt{m+m_2})\|Z_0\|_{L^2(\Omega,H^1)})^{p/(p-1)}},\frac1{C \|Z_0\|_{L^2(\Omega, H^1)}^{(2p)/(p-2)}} \Big)
$$ 
the equation \eqref{eqonZ} with initial datum $Z_0$ admits a unique solution $Z$ in $L^2(\Omega,\mathcal L_T)$. This solution satisfies 
$$
\|Z\|_{L^2(\Omega,\mathcal L_T)} \leq C \|Z_0\|_{L^2(\Omega, H^1)}.
$$
\end{proposition}

\begin{proof} By the Duhamel formula, $Z$ is the fixed point of
$$
A(Z) = S(t) Z_0 + \int_{0}^t S(t-\tau) \Big(\Big( \E(|Z|^2) + 2\Re ( \E(\overline Y Z)) \Big) (Y+Z)\Big).
$$
We proceed with a contraction argument.

Thanks to \eqref{strR}, we have
$$
\|A(Z)\|_{\mathcal L_T} \leq C \Big( \|Z_0\|_{H^1} + \int_{0}^T \|\Big( \E(|Z|^2) + 2\Re ( \E(\overline Y Z)) \Big) (Y+Z)\|_{H^1}d\tau.
$$
Taking the $L^2$ norm in probability yields
$$
\|A(Z)\|_{L^2(\Omega,\mathcal L_T)} \leq C \Big( \|Z_0\|_{L^2(\Omega,H^1)} + \int_{0}^T \|\Big( \E(|Z|^2) + 2\Re ( \E(\overline Y Z)) \Big) (Y+Z)\|_{L^2(\Omega),H^1)}.
$$
We use the definition of the $L^2(\Omega),H^1)$ norm of $g$ such as the $L^2(\Omega \times \R^d)$ norm of $g$ added to the $L^2(\Omega \times \R^d)$ norm of $\grad g$.

For the part not containing any derivative, we start by taking the $L^2$ norm in probability, which yields
\begin{multline*}
\|\Big( \E(|Z|^2) + 2\Re ( \E(\overline Y Z)) \Big) (Y+Z)\|_{L^2(\Omega\times \R^d)} \leq \Big( \E(|Z|^2) + 2m \sqrt{\E(|Z|^2)} \Big)(m +\sqrt{\E(|Z|^2)} ) = \\
\Big( \|Z\|_{L^2(\Omega)}^2 + 2m  \|Z\|_{L^2(\Omega)}  \Big)(m + \|Z\|_{L^2(\Omega)} ).
\end{multline*}
Taking the $L^2$ norm in space yields
\begin{multline*}
\|\Big( \E(|Z|^2) + 2\Re ( \E(\overline Y Z)) \Big) (Y+Z)\|_{L^2(\Omega\times \R^d)} \leq \\
\Big( \|Z\|_{L^\infty(\R^d,L^2(\Omega))}\|Z\|_{L^2 (\R^d,L^2(\Omega))} + 2m  \|Z\|_{L^2 (\R^d,L^2(\Omega))}  \Big)(m + \|Z\|_{L^\infty(\R^d,L^2(\Omega))} ).
\end{multline*}
As a consequence of Minkowski's inequality, we have $\|\cdot \|_{L^\infty(\R^d, L^2(\Omega\times \R^d))} \leq \|\cdot \|_{ L^2(\Omega,L^\infty(\R^d))}$. Therefore,
\begin{multline*}
\int_{0}^T \|\Big( \E(|Z|^2) + 2\Re ( \E(\overline Y Z)) \Big) (Y+Z)\|_{L^2(\Omega\times \R^d)} d\tau \leq \\
C \|Z\|_{L^2(\Omega,\mathcal L_T)}\Big( T 2m^2 + T^{1-1/p}3m \|Z\|_{L^2(\Omega, \mathcal L_T)}+T^{1-2/p}\|Z\|_{L^2(\Omega, \mathcal L_T)}^2 \Big).
\end{multline*}

Let us deal with the part containing the derivatives. We look at the different terms under the integral. First, we differentiate
$$
\grad (\E(|Z|^2)Z) = E(|Z|^2)\grad Z + 2\Re (\E(\overline{\grad Z} Z)) Z
$$
and then we take the $L^2$ norm in probability, which yields
$$
\|\grad (\E(|Z|^2)Z)\|_{L^2(\Omega)} \leq 3 \|Z\|_{L^2(\Omega)}^2 \|\grad Z\|_{L^2(\Omega)}.
$$
For the other terms, we get
\begin{eqnarray*}
\|\grad(\E(|Z|^2) Y)\|_{L^2(\Omega)} &\leq &2 \sqrt m\|Z\|_{L^2(\Omega)} \|\grad Z\|_{L^2(\Omega)}+ \|Z\|_{L^2(\Omega)}^2 \sqrt{m_2},\\
\|\grad(2\Re(\E(\overline Y Z)) Z)\|_{L^2(\Omega)}& \leq & 2\sqrt{m_2} \|Z\|_{L^2(\Omega)}^2 +4\sqrt m \|Z\|_{L^2(\Omega)} \|\grad Z\|_{L^2(\Omega)}, \\
\|\grad(2\Re(\E(\overline Y Z)) Y)\|_{L^2(\Omega)}& \leq & 4\sqrt{mm_2} \|Z\|_{L^2(\Omega)} +2m  \|\grad Z\|_{L^2(\Omega)}.
\end{eqnarray*}
Summing up all these terms separating the ones containing derivatives of $Z$ and the other ones gives
\begin{eqnarray*}
\|\grad \Big(\Big( \E(|Z|^2) + 2\Re ( \E(\overline Y Z)) \Big) (Y+Z)\Big)\|_{L^2(\Omega)} & \leq & \|\grad Z \|_{L^2(\Omega} \Big( 3 \|Z\|_{L^2(\Omega)}^2 + 6\sqrt m \|Z\|_{L^2(\Omega)} + 2m\Big) \\
 & & + \|Z\|_{L^2(\Omega)} \Big( 3\sqrt{m_2} \|Z\|_{L^2(\Omega} + 4 \sqrt{mm_2} \Big).
\end{eqnarray*}

We remark that $\|Z\|_{L^2(\Omega \times \R^d)} \leq \|Z\|_{L^2(\Omega, H^1)}$. Hence, taking the $L^2$ norm in space in the previous inequality gives
\begin{multline*}
\|\grad \Big(\Big( \E(|Z|^2) + 2\Re ( \E(\overline Y Z)) \Big) (Y+Z)\Big)\|_{L^2(\Omega\times \R^d)} \leq \\
\|Z\|_{L^2(\Omega, H^1)} \Big( 3 \|Z\|_{L^2(\Omega,L^\infty(\R^d))}^2 + (6\sqrt m+3\sqrt{m_2})  \|Z\|_{L^2(\Omega,L^\infty(\R^d))}+2m+4\sqrt{mm_2}\Big).
\end{multline*}
Integrating in time yields
\begin{multline*}
\int_{0}^T \|\grad \Big(\Big( \E(|Z|^2) + 2\Re ( \E(\overline Y Z)) \Big) (Y+Z)\Big)\|_{L^2(\Omega\times \R^d)} d\tau \leq \\
\|Z\|_{L^2,\mathcal L_T}\Big( 3 T^{1-2/p} \|Z\|_{L^2(\Omega, \mathcal L_T)}^2+ (6\sqrt m+3\sqrt{m_2}) T^{1-1/p} \|Z\|_{L^2(\Omega, \mathcal L_T)} +(2m+4\sqrt{mm_2})T\Big).
\end{multline*}

Going back to $A(Z)$, we have the estimate
\begin{multline*}
\|A(Z)\|_{L^2(\Omega,\mathcal L_T)}\leq C' \Big( \|Z_0\|_{L^2(\Omega,H^1)}+ \|Z\|_{L^2(\Omega,\mathcal L_T)}\Big(T(4m+4\sqrt{mm_2})  + \\
T^{1-1/p}(9\sqrt m+3\sqrt{m_2}) \|Z\|_{L^2(\Omega,\mathcal L_T)}+4T^{1-2/p}\|Z\|_{L^2(\Omega,\mathcal L_T)}^2\Big)\Big).
\end{multline*}

In conclusion if $\|Z\|_{L^2(\Omega,\mathcal L_T)} \leq 2C' \|Z_0\|_{L^2(\Omega,H^1)}$ then with 
$$
T =  \min \Big(1,\frac1{C(m+m_2)},\frac1{C(\sqrt{m+m_2})\|Z_0\|_{L^2(\Omega,H^1)})^{p/(p-1)}},\frac1{C \|Z_0\|_{L^2(\Omega, H^1)}^{(2p)/(p-2)}}\Big)
$$ 
for a constant $C$ big enough, we have
$$
\|A(Z)\|_{L^2(\Omega,\mathcal L_T)}\leq 2C' \|Z_0\|_{L^2(\Omega,H^1)}
$$
which means that the ball of $L^2(\Omega,\mathcal L_T)$ of radius $2C' \|Z_0\|_{L^2(\Omega,H^1)}$ is stable under $A$.

For the same reasons, we get that $A$ is contracting for appropriate times, which concludes the proof.\end{proof}

\subsection{Global well-posedness in the energy space for \texorpdfstring{$d\leq 3$}{d<=3}}

\begin{proposition}\label{prop-gwp} The equation \eqref{eqonZ} is globally well-posed in $H^1$. \end{proposition}

\begin{proof}We proceed with a modified energy method. Let
\begin{eqnarray*}
A &=& \frac12 \int_{\Omega\times \R^d} \overline Z (m-\lap )Z \\
B &=& \frac14 \int_{\R^d} \E(|Z|^2)^2\\
D &=& \int_{\R^d} \E(|Z|^2) \Re \E(\overline Z Y).
\end{eqnarray*}

Differentiating these quantities in time yields
\begin{eqnarray*}
\partial_t A &=& \Re \int_{\Omega\times \R^d} \partial_t \overline Z (m-\lap) Z \\
\partial_t B &=& \Re \int_{\Omega\times \R^d} \partial_t \overline Z \E(|Z|^2) Z\\
\partial_t D &=& \Re \int_{\Omega\times \R^d} \partial_t\Big( Z 2\Re \E(\overline Z Y) +Y \E(|Z|^2)\Big) + \Re\int_{\R^d}\E(|Z|^2) \E(\overline Z \partial_t Y).
\end{eqnarray*}
We deduce from that
$$
\partial_t (A+B+D) = \Re \int_{\Omega\times \R^d} \partial_t \overline Z \Big( i\partial_t Z - Y 2\Re \E(\overline Z Y)\Big)+\Re\int_{\R^d}\E(|Z|^2) \E(\overline Z \partial_t Y).
$$
Because of the real part we get
$$
\Re \int_{\Omega\times \R^d} \partial_t \overline Z \Big( i\partial_t Z - Y 2\Re \E(\overline Z Y)\Big)= -\Re \int_{\Omega\times \R^d} \partial_t \overline Z  Y 2\Re \E(\overline Z Y)
$$
and replacing $\partial_t \overline Z$ by its value
$$
-\Re \int_{\Omega\times \R^d} \partial_t \overline Z  Y 2\Re \E(\overline Z Y)= -\Re \int_{\Omega\times \R^d}i \Big((m-\lap) \overline Z + (E(|Z|^2 + 2\Re (\overline Z Y))(\overline{Z+Y})\Big) Y 2\Re \E(\overline Z Y)
$$
and again because of the real part
$$
-\Re \int_{\Omega\times \R^d} \partial_t \overline Z  Y 2\Re \E(\overline Z Y)= -\Re \int_{\Omega\times \R^d}i \Big((m-\lap) \overline Z + (E(|Z|^2 + 2\Re (\overline Z Y))\overline{Z}\Big) Y 2\Re \E(\overline Z Y).
$$
Returning to $A,B,D$, we get
$$
\partial_t (A+B+D) = -\Re \int_{\Omega\times \R^d}i \Big((m-\lap) \overline Z + (E(|Z|^2 + 2\Re (\overline Z Y))\overline{Z}\Big) Y 2\Re \E(\overline Z Y) + \Re\int_{\R^d}\E(|Z|^2) \E(\overline Z \partial_t Y).
$$
We estimate the different terms of the sum, we have 
\begin{eqnarray*}
\Big|-\Re \int_{\Omega\times \R^d}i Y 2\Re \E(\overline Z Y)(m-\lap) \overline Z\Big| &\leq &C(m,m_2) A\\
\Big|-\Re \int_{\Omega\times \R^d}i Y 2\Re \E(\overline Z Y)\E(|Z|^2)\overline{Z}\Big| & \leq & C(m) B\\
\Big|-\Re \int_{\Omega\times \R^d}i Y 2\Re \E(\overline Z Y)2\Re (\overline Z Y)\overline{Z}\Big| & \leq & C(m)A^{1/2}B^{1/2}\\
\Big| \Re\int_{\R^d}\E(|Z|^2) \E(\overline Z \partial_t Y)\Big| & \leq & C(m,m_2) A^{1/2}B^{1/2}.
\end{eqnarray*}

 The problem with this method is that $A+B+D$ does not control the $H^1$ norm. For this, we set
$$
E = \frac12 \int_{\Omega\times \R^d}|Z|^2.
$$
We have 
$$
|D| \leq \sqrt m B^{1/2}E^{1/2} \leq \frac12 B + 2m E.
$$
Hence setting $\mE = A+B+D+2m E$, we get $\mE \geq A + \frac12 B$. We prove now that $|\partial_t \mE| \leq C \mE$. Because of the previous computations 
$$
|\partial_t (A+B+D)|\leq C(m,m_2)\mE.
$$
We compute the derivative of $E$. We have
$$
\partial_t E = \Im \int_{\Omega\times \R^d} \overline Z i\partial_t Z = \Im \int_{\Omega\times \R^d} \overline Z \Big((m-\lap)  Z + (E(|Z|^2 + 2\Re (\overline Z Y))(Z+Y)\Big)
$$
and because of the imaginary part
$$
\partial_t E =\Im \int_{\Omega\times \R^d} \overline Z(E(|Z|^2 + 2\Re (\overline Z Y))Y.
$$
We get
$$
|\partial_t E| \leq \sqrt m E^{1/2}B^{1/2} + mE\leq C(m)(A^{1/2}B^{1/2} + A)\leq C(m) \mE.
$$
In conclusion, we get a bound for $\mE$ and thus for $A$, the $L^2(\Omega,H^1)$ norm of the solution, which implies global existence.
\end{proof}

\subsection{Local well-posedness in dimension 4 for small initial data}

In this subsection, we prove local well posedness for small initial data in $H^1$ in dimension $4$. We use a contraction argument in 
$$
\L_T = L^2(\Omega,\mathcal C ([-T,T],H^1(\R^4)))  \cap L^2(\Omega,L^3([-T,T],W^{1,3}(\R^4))).
$$
Thanks to Strichartz estimates, there exists $C$ such that for all $T \geq 0$, and all $g \in L^2(\Omega, H^1(\R^4))$, we have 
\begin{equation}\label{striR4}
\|S(t)g\|_{\L_T} \leq C\|g\|_{L^2(\Omega, H^1(\R^4))}.
\end{equation}

\begin{proposition} There exists $\varepsilon > 0$  such that for all $Z_0$ satisfying $\|Z_0\|_{L^2(\Omega, H^1(\R^4))} \leq \varepsilon$, the equation \eqref{eqonZ} admits a unique solution $Z$ in $\L_T$ for 
$$
T = \min\Big( \frac1{C(\sqrt m + \sqrt{m_2})^3} , \frac1{C(m+m_2)}\Big)
$$
with $C$ big enough.
Besides there exists $C$ such that 
$$
\|Z\|_{\L_T} \leq 2C\varepsilon
$$
and $Z$ depends continuously in $Z_0$.
\end{proposition}

\begin{proof} Let 
$$
A(Z) = S(t)Z_0 - i\int_{0}^t S(t-\tau) \Big( (\E(|Z|^2) + 2\Re \E(\overline Y Z))(Y+Z)\Big) d\tau.
$$
The solution $Z$ is the fixed point of $A$. We have, thanks to \eqref{striR4},
$$
\|A(Z)\|_{\L_T} \leq C \int_{-T}^T \big\| (\E(|Z|^2) + 2\Re \E(\overline Y Z))(Y+Z)\big\|_{L^2(\Omega, H^1)}
$$
which yields by a triangle inequality 
\begin{multline*}
\|A(Z)\|_{\L_T} \leq C  \Big( \|Z\|_{L^2(\Omega,H^1(\R^4))} +\int_{-T}^T \Big( \|\E(|Z|^2)Z\|_{L^2(\Omega, H^1)}+\|\E(|Z|^2)Y\|_{L^2(\Omega, H^1)} + \\
\|2\Re \E(\overline Y Z))Z\|_{L^2(\Omega, H^1)} +\|2\Re \E(\overline Y Z))Y\|_{L^2(\Omega, H^1)}\Big)d\tau \Big).
\end{multline*}
We have since $\|fg\|_{H^s} \lesssim \|(D^sf)g\|_{L^2}  + \|fD^sg\|_{L^2}$,
$$
\|\E(|Z|^2)Z\|_{L^2(\Omega, H^1)} \lesssim \|\,\|Z\|_{L^2(\Omega)}^2 \|DZ\|_{L^2(\Omega)} \|_{L^2(\R^4)}.
$$
Thanks to H\"older inequality, as $\frac13 + \frac16 = \frac12$, 
$$
\|\E(|Z|^2)Z\|_{L^2(\Omega, H^1)} \lesssim \|Z\|_{L^{12}(\R^4,L^2(\Omega))}^2\|DZ\|_{L^3(\R^4,L^2(\Omega))} 
$$
and as $12$ and $3$ are bigger than $2$, we can exchange the order of the norms,
$$
\|\E(|Z|^2)Z\|_{L^2(\Omega, H^1)} \lesssim \|Z\|_{L^2(\Omega,L^{12}(\R^4))}^2\|DZ\|_{L^2(\Omega,L^3(\R^4))} 
$$
and since $W^{1,3}(\R^4)$ is embedded in $L^{12}(\R^4)$,
$$
\|\E(|Z|^2)Z\|_{L^2(\Omega, H^1)} \lesssim \|Z\|_{L^2(\Omega,W^{1,3}(\R^4))}^3.
$$
Integrating in time yields
\begin{equation} \label{termcub}
\int_{-T}^T\|\E(|Z|^2)Z\|_{L^2(\Omega, H^1)} \lesssim \|Z\|_{\L_T}^3.
\end{equation}

Using that $\E(|Y|^2) = m$ and $\E(|DY|^2) = m+m_2$, we get for the quadratic terms
$$
\|\E(|Z|^2)Y\|_{L^2(\Omega,H^1(\R^4))} + \|2\Re \E(\overline Y Z))Z\|_{L^2(\Omega, H^1)} \lesssim (\sqrt m + \sqrt{m_2}) \|Z\|_{L^2(\Omega,L^6(\R^4))} \|Z\|_{L^2(\Omega, W^{1,3}(\R^4))}
$$
and as $W^{1,3}(\R^4)$ is embedded in $L^6(\R^4)$ and integrating in time, we get
\begin{equation} \label{termquad}
\int_{-T}^T\|\E(|Z|^2)Y\|_{L^2(\Omega,H^1(\R^4))} + \|2\Re \E(\overline Y Z))Z\|_{L^2(\Omega, H^1)} \lesssim (\sqrt m + \sqrt{m_2})T^{1/3}\|Z\|_{\L_T}^2.
\end{equation}

For the linear term we have 
$$
\|2\Re \E(\overline Y Z))Y\|_{L^2(\Omega, H^1)} \lesssim (m+m_2) \|Z\|_{L^2(\Omega, H^1)}
$$
which gives
\begin{equation} \label{termlin}
\int_{-T}^T \|2\Re \E(\overline Y Z))Y\|_{L^2(\Omega, H^1)} \lesssim ( m + m_2)T\|Z\|_{\L_T}.
\end{equation}

Summing \eqref{termcub}, \eqref{termquad}, \eqref{termlin}, we get
$$
\|A(Z)\|_{\L_T} \leq C \Big( \|Z\|_{L^2(\Omega,H^1(\R^4))}+\|Z\|_{\L_T}^3+(\sqrt m + \sqrt{m_2})T^{1/3}\|Z\|_{\L_T}^2+( m + m_2)T\|Z\|_{\L_T}\Big).
$$
Assuming that $\|Z\|_{L^2(\Omega,H^1(\R^4))} \leq \varepsilon$ with $\varepsilon$ such that
$$
2C \varepsilon \leq \frac1{2\sqrt C}
$$
and assuming 
$$
T = \min\Big( \frac1{8^3(\sqrt m + \sqrt{m_2})^3} , \frac1{8(m+m_2)}\Big),
$$
we get that the ball of $\L_T$ of radius $2C\varepsilon$ is stable under the map $A$.

What is more, for the same reasons, we get
\begin{multline*}
\|A(Z_1) - A(Z_2)\|_{\L_T} \leq C\Big( \|Z_1\|_{\L_T}^2 + \|Z_2\|_{\L_T}^2+\\
(\sqrt m + \sqrt{m_2})T^{1/3}(\|Z_1\|_{\L_T}+\|Z_2\|_{\L_T})+( m + m_2)T\Big)\|Z_1-Z_2\|_{\L_T}\Big).
\end{multline*}
Hence, for 
$$
T = \min\Big( \frac1{C(\sqrt m + \sqrt{m_2})^3} , \frac1{C(m+m_2)}\Big)
$$
with $C$ big enough and $\varepsilon$ small enough, we get that $A$ is contracting, which ensures existence and uniqueness of the fix point.

Finally, if $Z^1$ is the solution of \eqref{eqonZ} with initial datum $Z_0^1$ in the ball of radius $\varepsilon$, we have
$$
Z^1 = S(t)(Z^1_0 - Z_0) + A(Z^1)
$$
thus
$$
Z^1 - Z = S(t)(Z^1_0 - Z_0) + A(Z^1) - A(Z)
$$
and as $A$ is contracting,
$$
\|Z^1 - Z\|_{\L_T} \lesssim \|Z_0^1 - Z_0\|_{L^2(\Omega,H^1)}.
$$
\end{proof}

\section{Scattering and non-existence of localised equilibrium}\label{sec-scatt}

By copying the method of Lewin and Sabin in \cite{lewsabII}, it may be possible to prove scattering properties for the perturbed Hartree equation : 
$$
i\partial_t Z = (m-\lap)Z + w*\Big( \E(|Z|^2) + 2\Re \E(\overline Y Z)\Big)(Y+Z)
$$
with $w$ smooth enough. Scattering for the perturbed NLS \eqref{eqonZ} remains an open problem.

Nevertheless, one can prove scattering properties for \eqref{eqonrv}.

\subsection{Scattering for the defocusing equation}

We now prove scattering in $\R^3$.

We use Morawetz estimates in the spirit of \cite{GinibreVeloScat} and \cite{TzvVisScat}. We mention \cite{scattsyst} about scattering for a system of Schr\"odinger equations.

We follow the proof for decay estimates and scattering in \cite{taobook} from page 67 and onward. Because the computation for the linear part of the equation is the same up to constants, we will not insist on it and focus on the main difference, which is the non linearity.

\begin{proposition}\label{prop-scat} The equation \eqref{eqonrv} scatters in the sense that for all initial datum $X_0$ in \\ $L^2(\Omega,H^1(\R^3))$ there exists $X_{\pm \infty} \in L^2(\Omega,H^1(\R^3))$ such that
$$
\|X(t) - S(t) X_{\pm \infty}\|_{H^1(\R^3)} \rightarrow 0
$$
when $t$ goes to $\pm \infty$. By $X(t)$ we denote the solution of \eqref{eqonrv} with initial datum $X_0$ and by $S(t)$ the flow of the linear equation $\partial_t Z = -\lap Z$.
\end{proposition}

We start with decay estimates.

\begin{lemma} \label{lem-interMoraw} With the notations of the previous proposition we have that $X(t)$ belongs to $L^4(\R\times \R^3 , L^2(\Omega))$. In other terms, the quantity
$$
\int_{\R} dt \int_{\R^3} dx \E(|X(t,x)|^2)^2
$$
is finite.
\end{lemma}

\begin{proof} We start from the fact that $X$ satisfies a conservation law written : for all $j= 0,1,2,3$,
$$
\partial_t T_{j0} = \sum_{k=1}^3 \partial_{x_k} T_{jk}
$$
with $T_{00} = \E(|X|^2)$, $T_{0j} = T_{j0} = -2\Im \E(\overline X \partial_{x_j}X)$ for $j>0$ and for $j,k > 0$,
$$
T_{jk} = 2 \Re(\E( \partial_{x_j} X \overline{\partial_{x_k}X})) - \frac12\delta_j^k \lap(\E(|X|^2))+ \delta_j^k \E(|X|^2)^2.
$$
Indeed, for $j=0$, we have 
$$
\partial_t T_{00} = 2\Re \E (\partial_t X \overline X) = 2\Im \E (i\partial_t X \overline X) = -2\Im\E( \lap X \overline X) + 2\Im \E ( \E(|X|^2) |X|^2)
$$
Because of the imaginary part, the second term is $0$. Besides, we have 
$$
\partial_{x_k} T_{0k} = -2\Im \E(|\partial_{x_k}X|^2) - 2\Im \E( \overline X \partial_{x_k}^2 X).
$$
Because of the imaginary part the first term is $0$ and summing over $k$ yields
$$
\partial_t T_{00} = \sum_{k=1}^3 \partial_{x_k} T_{0k}.
$$
For $j> 0$, we have 
$$
\partial_t T_{j0} = -2\Im \E(\partial_t \overline X \partial_{x_j}X)-2\Im \E(\overline X \partial_t\partial_{x_j}X) = -2\Re \E( \overline{ i \partial_t X} \partial_{x_j}X)+2\Re \E(\overline X \partial_{x_j}i\partial_t X).
$$
As $X$ solves \eqref{eqonrv}, we get
$$
\partial_t T_{j0} = 2\Re \E( \overline{ \lap X} \partial_{x_j}X)-2\Re \E(\overline X \partial_{x_j}\lap X) -2\Re \E( \overline{ \E (|X|^2)X} \partial_{x_j}X)+2\Re \E(\overline X \partial_{x_j}(\E (|X|^2)X)).
$$

For the same reasons as in the deterministic case, we have for the terms involving only the linear part of the equation,
$$
 2\Re \E( \overline{ \lap X} \partial_{x_j}X)-2\Re \E(\overline X \partial_{x_j}\lap X) = \sum_{k=1}^3 \partial_{x_k} \Big(  2 \Re( \E(\partial_{x_j} X \overline{\partial_{x_k}X})) - \frac12\delta_j^k \lap(\E(|X|^2))\Big).
$$

For the term involving the non-linearity, we have that 
$$
-2\Re \E( \overline{ \E (|X|^2)X} \partial_{x_j}X)+2\Re \E(\overline X \partial_{x_j}(\E (|X|^2)X)) = 4  \E(|X|^2) \Re \E( (\partial_{x_j}X) \overline X)
$$
and
$$
\partial_{x_k} \delta_j^k \E(|X|^2)^2 =  \delta_j^k 4 \E(|X|^2) \Re \E( (\partial_{x_j}\overline X) X).
$$
Summing over $k$ yields
$$
\partial_t T_{j0} = \sum_{k=1}^3 \partial_{x_k} T_{jk}.
$$

Thanks to this structure, we repeat the usual computation to get
$$
\partial_t \int_{\R^3 \times \Omega } \sum_j \frac{x_j}{|x|} \Im (\overline X \partial_{x_j} X)dx = \int_{\R^3 \times \Omega} \frac{|\grad_0 X(x)|^2}{|x|}dx+\int_{\R^3} \frac{\E(|X|^2)^2}{|x|}dx
$$
where $\grad_y$ is the angular part of the gradient centred in $y$ and thus $\grad_0$ is merely the angular gradient. We get the Morawetz estimate :
$$
\int_{\R\times \R^3} \frac{\E(|X|^2)^2}{|x|}dxdt \leq \sup_{t\in \R} \|X(t)\|_{H^1(\R^3)}^2 < \infty.
$$
Translating the last equality by $y$, we get
\begin{eqnarray*}
\partial_t \int_{\R^3 \times \Omega } \sum_j \frac{x_j-y_j}{|x-y|} \Im (\overline X(x) \partial_{x_j} X(x))dx &= &\int_{\R^3 \times \Omega} \frac{|\grad_y X(x)|^2}{|x-y|}dx+\\
&& \int_{\R^3} \frac{\E(|X(x)|^2)^2}{|x-y|}dx+\pi \E(|X(y)|^2).
\end{eqnarray*}
Finally, multiplying by $\E(|X(y)|^2)$ and integrating over $y$, we get
$$
\partial_t \int_{\R^3 \times \R^3 } \sum_j \frac{x_j-y_j}{|x-y|} \E(|X(y)|^2)|\Im \E(\overline X(x) \partial_{x_j} X(x))dxdy = I+II+III+IV
$$
with
\begin{eqnarray*}
I&=& \int_{\R^3 \times \R^3} \E(|X(y)|^2)\frac{|\grad_y X(x)|^2}{|x-y|}dxdy\\
II&=&\int_{\R^3\times \R^3} \E(|X(y)|^2)\frac{\E(|X(x)|^2)^2}{|x-y|}dxdy\\
III &=& \pi \int_{\R^3}\E(|X(y)|^2)^2 dy \\
IV &=& \int_{\R^3 \times \R^3 } \sum_j \frac{x_j-y_j}{|x-y|} \partial_t \Big( \E(|X(y)|^2) \Big) |\Im \E(\overline X(x) \partial_{x_j} X(x))dxdy .
\end{eqnarray*}
The terms $I$ and $II$ are non negative. The term $III$ is the one we want to estimate. For the same structural reasons as in the deterministic case, the term $IV$ is controlled by $I$. Hence, we get that 
$$
III \leq \partial_t \int_{\R^3 \times \R^3 } \sum_j \frac{x_j-y_j}{|x-y|} \E(|X(y)|^2)|\Im \E(\overline X(x) \partial_{x_j} X(x))dx dy
$$
and we get the interaction Morawetz estimate 
$$
\int_{\R\times \R^3} \E(|X|^2)^2dxdt \leq \sup_{t\in \R} \|X(t)\|_{H^1(\R^3)}^4 < \infty
$$
which concludes the proof.
\end{proof}

Let $I$ be an interval of $\R$. We call $\L_I$ the space 
$$
\L_I = L^{10}(I,L^{10}(\R^3)) \cap L^{10/3}(I,W^{1,10/3}(\R^3)).
$$

\begin{lemma}\label{lem-L5} With the notations of Proposition \ref{prop-scat}, we have $X \in L^2(\Omega, \L_\R)$. \end{lemma}

\begin{proof} Let $I= [t_1,t_2]$. For all $t \in T$, the Duhamel formula of \eqref{eqonrv} writes
$$
X(t) = S(t-t_1)X(t_1) - i \int_{t_1}^t S(t-\tau) \Big( \E(|X(\tau)|^2)X(\tau) d\tau.
$$
We have that $W^{1,30/13}(\R^3)$ is embedded in $L^{10}(\R^3)$ by Sobolev's embedding, and $(10,\frac{30}{13})$ and $(\frac{10}{3},\frac{10}{3})$ are admissible for the Schr\"odinger dispersion in dimension $3$, since
$$
\frac{2}{10} + \frac{3}{30/13}= \frac{15}{10} = \frac32 \mbox{ and } \frac{2}{10/3} + \frac{3}{10/3} = \frac32.
$$
Besides, $\frac{10}{7}$ is the conjugate of $\frac{10}{3}$ hence, thanks to Strichartz estimates and a $TT^*$ argument
$$
\|X\|_{\L_I} \leq C \|X(t_1)\|_{H^1} + C\|\E(|X|^2) X\|_{L^{10/7}(I,W^{1,10/7}(\R^3))}  .
$$
We use the fact that $\frac{10}{7} \leq 2$ to apply Minkowski inequality and get
$$
\|X\|_{L^2(\Omega,\L_I)} \leq C \|X(t_1)\|_{L^2(\Omega,H^1)} + C\|D(\E(|X|^2) X)\|_{L^{10/7}(I,L^{10/7}(\R^3,L^2(\Omega)))}  .
$$
Distributing the derivative, we get
$$
\|D(\E(|X|^2) X)\|_{L^2(\Omega)} \leq \|D X\|_{L^2(\Omega)}  \| X\|_{L^2(\Omega)}^2.
$$
Using H\"older's inequality with $2\frac15 + \frac{3}{10}= \frac{7}{10}$, we get
$$
\|D(\E(|X|^2) X)\|_{L^{10/7}(I,L^{10/7}(\R^3,L^2(\Omega)))} \leq \|D X\|_{L^{10/3}(I\times \R^3,L^2(\Omega))}  \| X\|_{L^5(I\times \R^3,L^2(\Omega))}^2.
$$
Using again Minkowski's inequality as $\frac{10}{3} \geq 2$, we get
$$
\|D X\|_{L^{10/3}(I\times \R^3,L^2(\Omega))} \leq \|X\|_{L^2(\Omega, L^{10/3}(I, W^{1,10/3}(\R^3)))}  \leq \|X\|_{L^2(\Omega,\L_I)}.
$$
Using that $5$ lies between $4$ and $10$, we get
$$
\| X\|_{L^5(I\times \R^3,L^2(\Omega))} \leq \| X\|_{L^4(I\times \R^3,L^2(\Omega))}^{2/3}\| X\|_{L^{10}(I\times \R^3,L^2(\Omega))}^{1/3}.
$$
Using once more Minkowski's inequality and the definition of $\L_I$, we have 
$$
\| X\|_{L^{10}(I\times \R^3,L^2(\Omega))} \leq \|X\|_{L^2(\Omega,\L_I)}.
$$

Besides, we use that thanks to the conservation of the energy the quantity $\|X(t_1)\|_{L^2(\Omega,H^1)}$ is bounded uniformly in $t_1$ by a quantity $\mE_0$.

Summing up, we get
$$
\|X\|_{L^2(\Omega,\L_I)} \leq C \mE_0 + C\|X\|_{L^2(\Omega,\L_I)}^{5/3}\|X\|_{L^4(I \times \R^3,L^2(\Omega))}^{4/3} .
$$
Let $\varepsilon =  \mE_0^{-1/2} (2C)^{-5/2}$. As, by Lemma \ref{lem-interMoraw}
$$
\|X\|_{L^4(\R \times \R^3,L^2(\Omega))} = \Big( \int_{\R}dt \int_{\R^3}dx \E(|X(t,x)|^2)^2 \Big)^{1/4}
$$
is finite, there exist a finite family of intervals $(I_j)_{1\leq j\leq r}$ such that
$$
\bigcup_{i=1}^rI_j = \R \mbox{ and for all } j\, , \|X\|_{L^4(I_j \times \R^3,L^2(\Omega))} \leq \varepsilon.
$$
Therefore, for all $j$, we get 
$$
\|X\|_{L^2(\Omega,\L_{I_j})} \leq C \mE_0 + C\|X\|_{L^2(\Omega,\L_{I_j})}^{5/3}\varepsilon^{4/3} .
$$
This choice of $\varepsilon$ implies $\|X\|_{L^2(\Omega,\L_{I_j})} \leq 2C\mE_0$. Summing over $j$ yields 
$$
\|X\|_{L^2(\Omega,\L_{\R})} \lesssim \mE_0 < \infty
$$
hence the result. \end{proof}

We describe $X_{\pm \infty}$. 

\begin{lemma} \label{lem-Xinfty} Let 
$$
X_{\pm \infty} = X_0 - i \int_{0}^{\pm \infty} S(-\tau) \E(|X(\tau)|^2)X(\tau) d\tau.
$$
The maps $X_{\pm \infty}$ belong to $L^2(\Omega,H^1(\R^3))$.
\end{lemma}

\begin{proof} First, $X_0 \in L^2(\Omega,H^1(\R^3))$. Then, thanks to Strichartz estimates and a $T^*$ argument, we get
$$
\big\| \int_{0}^{\pm \infty} S(-\tau) \E(|X(\tau)|^2)X(\tau) d\tau\big\|_{L^2(\Omega,H^1(\R^3))} \leq  C  \|D(\E(|X|^2)X)\|_{L^2(\Omega,L^{10/7}(\R\times \R^3))} .
$$
With the same computation as previously, we get
$$
\big\| \int_{0}^{\pm \infty} S(-\tau) \E(|X(\tau)|^2)X(\tau) d\tau\big\|_{L^2(\Omega,H^1(\R^3))} \leq C \|X\|_{L^2(\Omega ,L^5(\R\times \R^3))}^2 \|X\|_{L^2(\Omega,L^{10/3}(\R,W^{1,10/3}(\R^3))) }
$$
which is finite by interpolation. \end{proof}

\begin{proof}[Proof of Proposition \ref{prop-scat}.] We focus on $+\infty$. We have 
\begin{multline*}
\|X(t) - S(t)X_{+ \infty}\|_{H^1(\R^3)} = \big\| \int_{t}^\infty S(t-\tau) \E(|X(\tau)|^2X(\tau)d\tau)\big\|_{H^1(\R^3)}\\
\leq C \|1_{\tau \geq t}X\|_{L^2(\Omega ,L^5(\R\times \R^3))}^2 \|1_{\tau \geq t}X\|_{L^2(\Omega,L^{10/3}(\R,W^{1,10/3}(\R^3))) }
\end{multline*}
which goes to $0$ as $t$ goes to $\infty$. We use the dominated convergence theorem to handle the $L^2(\Omega)$ norm.\end{proof}

\subsection{Lack of localised equilibrium}

\begin{proposition} Let $Y$ be a solution of \eqref{eqonrv} whose law is invariant in time. Assume that $Y(t=0)$ belongs to $L^2(\Omega,H^1(\R^3))$. Then $Y = 0$.
\end{proposition}

\begin{proof}
Indeed, if $Y$ is in $L^2(\Omega,H^1(\R^3))$ then thanks to lemma \ref{lem-L5}, $Y$ belongs to $L^2(\Omega, L^{10}(\R\times \R^3))$ which is continuously embedded in $L^{10}(\R\times \R^3, L^2(\Omega))$. We have 
$$
\|Y\|_{L^{10}(\R\times \R^3, L^2(\Omega))}^{10} = \int_{\R} dt \int_{\R^3}dx \E(|Y(t,x)|^2)^5.
$$
Because the law of $Y$ does not depend on time, we have that $\E(|Y(t,x)|^2)^5$ is a map $\varphi(x)$ which does not depend on time. Hence $\int_{\R^3}dx \E(|Y(t,x)|^2)^5$ is a constant and thus, for it to be integrable, it has to be $0$, which ensures that $Y=0$.
\end{proof}

\section{On the focusing case}\label{sec-focus}

Up to now, we have only considered the defocusing case but we can now consider the focusing equation :
\begin{equation}\label{eqonrvfoc}
i\partial_t X = -\lap X - \E(|X|^2)X
\end{equation}
in $\R^d$, $d\leq 3$. 

First of all, this equation is locally well-posed for initial data taken in $H^1(\R^d)$, $d\leq 3$.

Besides, we remark that \eqref{eqonrvfoc} has stationary solutions. Let $Q$ be a stationary solution of $i\partial_t u = - \lap u - |u|^2 u$ and $X$ be a random variable such that the probability that $X = Q$ is $1$. Then $X$ is a stationary solution of \eqref{eqonrvfoc}.

We prove the existence of blow-up solutions for the focusing equation.

We proceed with a viriel method. We prove that 
\begin{equation}\label{defvir}
V(t) = \int_{\Omega\times \R^d} |x|^2|X|^2
\end{equation}
is well-defined on $[0,T]$ as long as the solution $X$ of \eqref{eqonrvfoc} is well posed on $[0,T]$.

\begin{lemma}\label{lem-computederone} Let $\varphi$ be a non negative $\mathcal C^1 $ function on $\R^d$ with compact support. We have 
$$
\partial_t \Big( \int_{\Omega \times \R^d} \varphi(x) |X|^2\Big) = 2 \Im \int_{\Omega \times \R^d} \grad \varphi \overline X \grad X.
$$
\end{lemma}

\begin{proof} The computation is the same as in the deterministic case, which yields
$$
\partial_t \Big( \int_{\Omega \times \R^d} \varphi(x) |X|^2 \Big)= 2 \Im \int_{\Omega\times \R^d} \varphi(x) \overline X (-\lap X + \E(|X|^2X).
$$
We have $\varphi(x) \overline X \E(|X|^2)X \in \R$ thus we keep only
$$
2 \Im \int_{\Omega\times \R^d} \varphi(x) \overline X (-\lap X)
$$
and with an integration by parts we get
$$
2 \Im \int_{\Omega\times \R^d}\grad( \varphi(x) \overline X) \grad X
$$
and by developing the gradient and seeing that $ \varphi |\grad X|^2 \in \R$, we get the result. \end{proof}

Let $\varphi$ the specific function such that 
$$
\varphi(x) = \left \lbrace{\begin{array}{lll}
|x|^2 & \mbox{ if } |x|\leq 1 \\
e^{1-1/(|x|-2)^2} & \mbox{ if } |x|\in [1,2]\\
0 & \mbox{ otherwise. }
\end{array}}\right.
$$
We have $\varphi \in \mathcal C^1$ with compact support and there exists $C$ such that for all $x\in \R^d$, $|\grad \varphi (x)|^2 \leq C\varphi(x)$.

\begin{lemma}\label{lem-welldefined} Assuming that $V(t=0)$ is well-defined, the Viriel $V(t)$ is well-defined on $[0,T]$ as long as the solution $X$ of \eqref{eqonrvfoc} is well posed on $[0,T]$.\end{lemma}

\begin{proof} For all $R > 0$ let $\varphi_R(x) = R^2 \varphi(\frac{x}{R})$. We have 
$$
\int_{|x|\leq R} |x|^2 |X|^2 \leq \int \varphi_R (x) |X|^2 .
$$
We apply the last lemma to get
$$
\partial_t \Big( \int \varphi_R(x) |X|^2 \Big) =  2 \Im \int_{\Omega \times \R^d} \grad \varphi_R \overline X \grad X.
$$
We apply Cauchy-Schwartz inequality to get
$$
\Big| \partial_t \Big( \int \varphi_R(x) |X|^2 \Big) \Big| \leq 2 \|\grad \varphi_R \overline X\|\, \|X(t)\|_{H^1}.
$$
We use that $|\grad{\varphi_R} (x)|^2=| R \grad \varphi(\frac{x}{R})|^2 \leq C R^2 \varphi (\frac{x}{R}) = \varphi_R(x)$ to get
$$
\Big| \partial_t \Big( \int \varphi_R(x) |X|^2 \Big) \Big| \leq C \Big(\int \varphi_R(x) |X|^2 \Big)^{1/2} \|X(t)\|_{H^1}
$$
From which we deduce 
$$
 \Big(\int \varphi_R(x) |X|^2 \Big)^{1/2} \leq  V(t=0)^{1/2} + C \int_{0}^T \|X(\tau)\|_{H^1} d\tau.
$$
As the right hand side is bounded uniformly in $R$, we get the result. \end{proof}

We compute the second derivative of $V$. 

\begin{lemma}We have, where $V$ is defined
$$
\partial_t^2 V(t) \leq 16 \mE(X_0).
$$
\end{lemma}

\begin{proof} We have, thanks to Lemma \ref{lem-computederone}
$$
\partial_t V = 4\Im \int_{\Omega\times \R^d  } x\grad X \overline X.
$$
We differentiate it a second time to get
$$
\partial_t^2 V =  I + II
$$
with
$$
I= 4 \Re \int_{\Omega \times \R^d} x\Big(  \overline{i\partial_t X} \grad X \Big) \mbox{ and }II= -4 \Re \int_{\Omega \times \R^d} x\Big( \overline X \grad (i\partial_t X)\Big) .
$$
By integration by parts, we get that $II$ is given by
$$
4\Re d \int \overline X (i\partial_t X) + 4\Re \int x \grad \overline X (i\partial_t X) 
$$
and thus, by replacing $i\partial_t X$ by its value,
$$
\partial_t^2 V(t) = 4d  \int_{\Omega \times \R^d} \overline X (-\lap) X + 4d \int_{\R^d} \E(|X|^2)^2  +2 I.
$$

We compute $I$.  By replacing $i\partial_t X$ by its value, we get
$$
I = I.1 + I.2 
$$
with
$$
I.1 = 4\Re \int_{\Omega \times \R^d} x\grad X ( -\lap \overline X ) \mbox{ and } I.2 = 4\Re \int_{\Omega \times \R^d} x\grad X (- \E(|X|^2) \overline X\Big) 
$$
The computation for $I.1$ is the same as in the deterministic case, and we get 
$$
I.1 = (2d-4) \int_{\Omega \times \R^d} X\lap \overline X.
$$

The computation for $I.2$ requires to take into account the probability. We replace the gradient by partial derivatives to get
$$
I.2 =- 4\Re\sum_j \int_{\Omega\times \R^d} x_j \E(|X|^2)  \overline X \partial_j X
$$
where $\partial_j = \partial_{x_j}$. We replace the integral in $\Omega$ by the expectation $\E$ to get
$$
I.2 = -4 \Re \sum_j \int_{\R^d}x_j \E(|X|^2) \E ( \overline X \partial_j X).
$$
We remark that $\partial_j \E(|X|^2)^2 = 4\Re \E(|X|^2)\E(\overline X \partial_j X)$ such that
$$
I.2 = -\sum_j \int_{\R^d}x_j \partial_j \E(|X|^2)^2
$$
and by integration by parts
$$
I.2 = d \int_{\R^d} \E(|X|^2)^2 .
$$

Summing up, we get
$$
\partial_t^2 V(t) = 8 \int_{\Omega\times \R^d} \overline X (-\lap) X -2d \int_{\R^d} \E(|X|^2)^2
$$
and for $d \geq 2$,
$$
\partial_t^2 V \leq 16 \mE(X(t)) = 16 \mE(X_0).
$$
\end{proof}

\begin{proposition}\label{prop-blowup}If $X_0 \in L^2(\Omega, H^1(\R^d))$ is such that
 $V(t=0)$ is finite and $\mE(X_0) < 0$, then the solution of \eqref{eqonrvfoc} blows up at finite time.\end{proposition}

\section{Incidence at the operator level}\label{sec-incidence}

\subsection{Incidence at the operator level on the sphere and torus}

In this section, we prove the global well-posedness of \eqref{eqonop} on the sphere and torus.

Let $M \in \{\S^2,\S^3,\T^2,\T^3\}$.

\subsubsection{Uniqueness of laws}

In this subsection we prove that two solutions of \eqref{eqonrv} whose initial data have the same law have also the same law. For this setting, it is relevant to use Subsection \ref{subsec-interplaw}. Nevertheless, since the following technique is easier to expose in this setting rather than for the perturbed equation and since we require it for the perturbed equation, we choose to present it here.

\begin{lemma}\label{lem-lemlaw1} Let $X(t)$ be a solution of \eqref{eqonrv} with initial datum $X_0$ defined on the probability space $(\Omega, \mathcal F, P)$ and belonging to $L^2(\Omega, H^1(M))$. Let $(\omega_1,\omega_2) \in \Omega^2$. If $X_0(\omega_1) = X_0(\omega_2)$, then at all times $t$, $X(t,\omega_1) = X(t,\omega_2)$.
\end{lemma}

\begin{proof} Let $\varphi(t,x) = \E(|X(t,x)|^2)$. Both $X(t,\omega_1)$ and $X(t,\omega_2)$ are solutions of
$$
i\partial_t u = -\lap u + \varphi(t,x) u
$$
with the same initial datum $u_0 = X_0(\omega_1) = X_0(\omega_2)$. In view of the previous sections, this ensures that $X(t,\omega_1) = X(t,\omega_2)$.
\end{proof}

\begin{definition}\label{def-deflaw} Given an  initial datum $X_0$ defined on the probability space $(\Omega, \mathcal F, P)$ and belonging to $L^2(\Omega, H^1(M))$, let $\sim_P$ be the equivalence relation on $\Omega$ defined as 
$$
\omega_1 \sim_P \omega_2 \Leftrightarrow X_0(\omega_1) = X_0(\omega_2).
$$
Let $(\Omega', \mathcal F', P')$ be the probability space $(\Omega, \mathcal F, P)$ quotiented by $\sim_P$, that is
\begin{eqnarray*}
\Omega' &=& \{cl(\omega)\, |\, \omega \in \Omega\},\\
\mathcal F' &=& \{cl( X_0^{-1}(A)) \,|\, A \textrm{ mesurable in } H^1(M)\}, \\
\forall C \in \mathcal F', \; P'(C)& =& P\Big(\bigcup_{c\in C} c\Big)
\end{eqnarray*}
where
\begin{eqnarray*}
cl(\omega) &=& \{\omega' \in \Omega \,|\, \omega' \sim_P \omega\}\\
cl(A) &=& \{cl(\omega) \,|\, \omega \in A\}.
\end{eqnarray*}
Finally, let $X'(t)$ be the random variable defined on $(\Omega', \mathcal F', P')$ and belonging to $L^2(\Omega',H^1(M))$  as $X'(t)(cl(\omega)) = X(t) (\omega)$.
\end{definition}

\begin{remark} The measure $P'$ is well-defined on $\mathcal F'$ and 
$$
P'(cl( X_0^{-1}(A)))= P(X_0^{-1}(A)).
$$
Indeed, if $\omega \in X_0^{-1}(A)$, then $cl(\omega) \subseteq X_0^{-1}(A)$.

The random variable $X'(t)$ is defined without ambiguity thanks to Lemma \ref{lem-lemlaw1}. It belongs to $L^2(\Omega', H^1(M))$ since 
$$
\E(\|X'(t)\|_{H^1}^2) = \int_{\R^+} P'(X'(t)^{-1}(B_{H^1}(0, \sqrt \lambda)^c))d\lambda
$$
where ${}^c$ stands for the complementary set. Given the definition of $P'$, this yields 
$$
\E(\|X'(t)\|_{H^1}^2) = \int_{\R^+} P(X(t)^{-1}(B_{H^1}(0, \sqrt \lambda)^c))d\lambda = \E(\|X(t)\|_{H^1}^2) < \infty.
$$
\end{remark}

\begin{lemma}\label{lem-lemlaw3} The law of $X'(t)$ is the same as the law of $X(t)$. \end{lemma}

\begin{proof} This is due that for all measurable $A$ set in $H^1(M)$, we have 
$$
X'(t)^{-1}(A) = cl(X(t)^{-1}(A))
$$
and due to the definition of $P$. Note that $X(t)^{-1}(A)$ is measurable in $\Omega$ and $X'(t)^{-1}(A)$ measurable in $\Omega'$ because the flow of \eqref{eqonrv} is continuous.\end{proof}

\begin{lemma}\label{lem-law} Let $X_1$ and $X_2$ be solutions to \eqref{eqonrv} with initial datum $X_{1,0} \in \L^2(\Omega_1,H^1(M))$ and $X_{2,0}\in L^2(\Omega_2,H^1(M))$ which have the same law. Then, for all $t$, $X_1(t)$ and $X_2(t)$ have the same law.\end{lemma}

\begin{proof}Thanks to Lemma \ref{lem-lemlaw3} and using the same notations, we can consider the random variables $X'_1$ and $X'_2$ instead of $X_1$ and $X_2$. Let $\varphi$ be the map from $\Omega_1'$ to $\Omega_2'$ defined as 
$$
\varphi(X_{0,1}^{-1}(\{u_0\}) )= X_{0,2}^{-1}(\{u_0\})
$$
for all $u_0 \in H^1(M)$. 

By construction, $X_{1,0}'  = X_{2,0}' \circ\varphi$ and $P'_2$  is the image measure of $P_1'$ under $\varphi$. 

By uniqueness of the flow of \eqref{eqonrv}, $X'_2(t) \circ \varphi = X'_1(t)$ and since $\varphi$ preserves the measure the law of $X'_2(t)$ is the same has the one of $X'_1(t)$. Therefore, thanks to Lemma \ref{lem-lemlaw3}, $X_1(t)$ and $X_2(t)$ have the same law.
\end{proof}

\subsubsection{Gaussian variables}

In this subsection, we prove that if $X_0$ is a Gaussian variable, then so is $X(t)$ the solution of \eqref{eqonrv} with initial datum $X_0$. What is more, we prove that if $\gamma$ is a solution of \eqref{eqonop} then there exists a Gaussian variable with covariance $\gamma$ that is a solution of \eqref{eqonrv}.

\begin{lemma}\label{lem-law2} Let $X_0$ be a Gaussian process of covariance $\gamma_0$ with $\textrm{Tr}((1-\lap)\gamma_0) < \infty$. Let $X(t)$ be the solution of \eqref{eqonrv} with initial datum $X_0$ then $X(t)$ is a Gaussian process.
\end{lemma}

\begin{proof} Write $\varphi(t,x) = \E(|X(t,x)|^2)$. By Propositions \ref{prop-lwpS}, and \ref{prop-lwpM3} one gets that the equation
$$
i\partial_t u = -\lap u + \varphi u
$$
is well-posed in $H^1$. Let $U(t)$ be the flow of this equation, it is linear and continuous on $H^1$. Let $\lambda \in H^{-1}$. Since by uniqueness of the flow we have $X(t) = U(t)X_0$, 
$$
\E(e^{i\an{\lambda, X(t)}}) = \E(e^{i\an{U^*(t)\lambda,X_0}}) = e^{-\an{U^*(t)\lambda,\gamma_0 U^*(t)\lambda}}= e^{-\an{\lambda,U(t)\gamma_0 U^*(t)\lambda}}.
$$
Since $U(t)\gamma_0U^*(t)$ is a positive operator, we get that $X(t)$ is a Gaussian process of covariance $U(t)\gamma_0U^*(t)$.
\end{proof}

\begin{definition}\label{def-dist} Let $\Sigma$ be the set of non negative operators $\gamma_0$ such that $\textrm{Tr}((1-\lap)\gamma_0)$ is finite on $M$. Let $d$ be the distance on this set defined as 
$$
d(\gamma_1,\gamma_2) = d_2(\nu_1,\nu_2)
$$
where $d_2$ is the Wasserstein distance defined in Remark \ref{rem-wass} and $\nu_i$ is the law of the Gaussian process with covariance $\gamma_i$. \end{definition}

\begin{remark} The Wasserstein distance may also be defined as 
$$
d_2(\nu_1,\nu_2) = \inf_{X_i \sim \gamma_i} \|X_1-X_2\|_{L^2(\Omega, H^1(M))}
$$
where $\sim$ stands for ``is a Gaussian random field of covariance''.

By $\Sigma$, we now denote the metric space $(\Sigma ,d)$.
\end{remark}

\begin{lemma}\label{lem-gau1} Assume that $\gamma \in \mathcal C(\R,\Sigma)$ is a solution to \eqref{eqonop}. Then there exists a probability space $\Omega$ and a Gaussian variable $X \in \mathcal C(\R,L^2(\Omega,H^1(M)))$ solution of \eqref{eqonrv} and of covariance $\gamma$. \end{lemma}

\begin{proof} Let $\varphi = \rho_{\gamma}$ and let $X$ be the solution to 
$$
i\partial_t X = (-\lap + \varphi) X
$$
with initial datum $X_0$ a Gaussian variable with covariance $\gamma_0$. Its covariance $\gamma_X$ is the unique solution to the linear equation 
$$
i\partial_t \gamma_X = [-\lap + \varphi, \gamma_X]
$$
with initial datum $\gamma_{X_0} = \gamma(t=0)$.

Indeed, let $U(t)$ the flow of the linear equation on $u$ : $i\partial_t u = (-\lap + \varphi) u$. The map $U$ is invertible. Hence, since,
$$
i\partial_t(U(t)^* \gamma_X U(t)) = 0
$$
the equation $i\partial_t \gamma_X = [-\lap + \varphi, \gamma_X]$ has a unique solution. 

Since $\gamma$ is also a solution to $i\partial_t \gamma_X = [-\lap + \varphi, \gamma_X]$ with initial datum $\gamma(t=0)$ we get that $\gamma_X = \gamma$, thus $\E(|X|^2) = \rho$, which ensures that $X$ is a sol to \eqref{eqonrv}.
\end{proof}

\subsubsection{Global well-posedness}

\begin{corollary}[of Proposition \ref{prop-gwpS}] \label{cor-gwpS}Let $\Psi$ be the map from $\Sigma$ to $\mathcal C( \R, \Sigma)$ such that $\Psi(t) (\gamma_0) = \gamma_{X(t)}$ where $X(t)$ is the solution to \eqref{eqonrv} with initial datum $X_0$ the Gaussian random process with covariance operator $\gamma_0$. The map $\Psi$ is well-defined, it defines a solution to \eqref{eqonop} and it is continuous for the distance $d$. Besides, $\Psi(t)\gamma_0$ is the unique solution to \eqref{eqonop} with initial datum $\gamma_0$.
\end{corollary}

\begin{proof} Because of Proposition \ref{prop-CI} we get that an initial datum $\gamma_0 \in \Sigma$ at the operator level such that $\textrm{Tr}((1-\lap)\gamma_0)$ is finite gives an initial datum at the level of random variables $X_0$ belonging to $L^2(\Omega, H^1(M))$. Thanks to Proposition \ref{prop-gwpS} we get a solution $X$ of \eqref{eqonrv}, and thanks to Proposition \ref{XtoG}, we get a solution $\gamma$ to the equation \eqref{eqonop}. We remark that thanks to Lemma \ref{lem-law}, one can take any Gaussian $X_0$ with covariance $\gamma_0$ as the law of $X(t)$ depends only on the law of $X_0$. Hence $\Psi(t)$ is well-defined. It is continuous in time and in the initial datum for the following reason : the distance between $\gamma_{X_1(t)}$ and $\gamma_{X_2(t)}$ is controlled by the norm of $X_1(t) - X_2(t)$. Indeed, $X_i$ is a Gaussian process by Lemma \ref{lem-law2}, therefore
$$
d(\gamma_1(t_1),\gamma_2(t_2)) \leq \|X_1(t_1) - X_2(t_2)\|_{H^1(M)} 
$$
where $\gamma_i(t)$ is equal to $\gamma_{X_i(t)}$.

The continuity of the solution $X(t)$ in both time and initial datum gives the result. Indeed, take any Gaussian process $X_i$ with covariance $\gamma_i$ and any couple of times $t_1,t_2$, we have 
$$
d(\gamma_1(t_1),\gamma_1(t_2)) \leq \|X_1(t_1) - X_1(t_2)\|_{L^2(\Omega, H^1)}\leq C(X_1)|t_1-t_2|^\alpha
$$
for some $\alpha > 0$ and $C(X_1) = C(\gamma_1)$ is a constant depending only on $\gamma_1$. What is more,
$$
d(\gamma_1(t_1),\gamma_2(t_1)) \leq \|X_1(t_1) - X_2(t_1)\|_{L^2(\Omega, H^1)}\leq C(t_1)\|X_1-X_2\|_{L^2(\Omega,H^1(M))}
$$
and by taking the infimum over the couples $(X_1,X_2)$ we get the result.

For the uniqueness of the solution, let $\gamma_1$ and $\gamma_2$ be two solutions of \eqref{eqonop} with the initial datum $\gamma_0$. For $i=1,2$, there exists $X_i(t)$ a solution of \eqref{eqonrv} which is a Gaussian variable of covariance $\gamma_i$. For $i = 1,2$, $X_i(t=0)$ is a Gaussian variable of covariance $\gamma_0$. Hence $X_1(t=0)$ and $X_2(t=0)$ have the same law. Therefore $X_1(t)$ and $X_2(t)$ too, which ensures that $\gamma_1 = \gamma_2$ and hence the uniqueness of the solution of \eqref{eqonop}.
\end{proof}

\begin{remark} One could rewrite the corollary \ref{cor-gwpS} as : the equation \eqref{eqonop} is globally well-posed in $\mathcal C(\R,\Sigma)$.
\end{remark}

\subsection{Global well-posedness on the Euclidean space}

Let $f$ be a bounded function on $\R^d$ such that $\an k f(k) \in L^2(\R^d)$ and let $Y_f$ be the equilibrium corresponding to $f$ that is 
$$
Y_f(t,x) = \int f(k) e^{i(m+k^2)t}e^{ikx}dW_k
$$
with $m = \int |f(k)|^2 dk$.

This random variable defines an equilibrium for \eqref{eqonrv} and the operator $\gamma_f = \gamma_{Y_f}$ is a stationary solution for \eqref{eqonop}. Indeed, $\gamma_{f}$ is the Fourier multiplier by $|f(k)|^2$. Hence it commutes with the Laplacian and $\rho_{\gamma_{f}} = m$.

In this section, we prove the global well-posedness of \eqref{eqonop} around equilibria $\gamma_f$, that is, we prove global well-posedness of the equation
\begin{equation}\label{eqonQ}
i\partial_t Q = [-\lap ,Q] + [\rho_{\gamma_f+Q},\gamma_f + Q]
\end{equation}
where $Q$ is not necessarily non-negative but $\gamma_f + Q$ is.

Let $M \in \{\R^2,\R^3\}$.

\subsubsection{Uniqueness of laws}

In this subsection we prove that two solutions of \eqref{eqonrv} whose initial data have the same law have also the same law.

\begin{lemma}\label{lem-lemlaw11} Let $X(t)$ be a solution of \eqref{eqonrv} with initial datum $X_0=Y_0+Z_0$ defined on the probability space $(\Omega, \mathcal F, P)$ and such that $Y_0$ has the same law as $Y_f(t=0)$ and such that $Z_0$ belongs to $L^2(\Omega, H^1(M))$. Write $Y(t) = e^{-it(-\lap + m)}Y_0$ and $X(t) = Y(t) + Z(t)$. Assume $Z \in \mathcal C(\R,(L^2(\Omega, H^1(M)))$. Let $(\omega_1,\omega_2) \in \Omega^2$. If $(Y_0,Z_0)(\omega_1) = (Y_0,Z_0)(\omega_2)$, then at all times $t$, $(Y,Z)(t,\omega_1) = (Y,Z)(t,\omega_2)$.
\end{lemma}

\begin{proof} First, if $Y_0(\omega_1) = Y_0(\omega_2)$, then $Y(t,\omega_1) = Y(t,\omega_2)$. Let $\varphi(t,x) = \E(|X(t,x)|^2)-m$. Both $Z(t,\omega_1)$ and $Z(t,\omega_2)$ are solutions of
$$
i\partial_t u = (m-\lap) u + \varphi(t,x) (u+Y(t,\omega_1))
$$
with the same initial datum $u_0 = Z_0(\omega_1) = Z_0(\omega_2)$. In view of the previous sections, this ensures that $Z(t,\omega_1) = Z(t,\omega_2)$.
\end{proof}

\begin{definition}\label{def-deflawlaw} Given an  initial datum $X_0=Y_0+Z_0$ defined on the probability space $(\Omega, \mathcal F, P)$ and with $Z_0$ belonging to $L^2(\Omega, H^1(M))$, and $Y_0$ with the same law as $Y_f(t=0)$. Let $\sim_P$ be the equivalence relation on $\Omega$ defined as 
$$
\omega_1 \sim_P \omega_2 \Leftrightarrow (Y_0,Z_0)(\omega_1) = (Y_0,Z_0)(\omega_2).
$$
Let $(\Omega', \mathcal F', P')$ be the probability space $(\Omega, \mathcal F, P)$ quotiented by $\sim_P$, and let $Z'(t)$ be the random variable defined on $(\Omega', \mathcal F', P')$ and belonging to $L^2(\Omega',H^1(M))$  as $Z'(t)(cl(\omega)) = Z(t) (\omega)$ and let $Y'(t)(cl(\omega)) = Y(t) (\omega)$
\end{definition}

\begin{remark}
The random variable $Z'(t)$ is defined without ambiguity thanks to Lemma \ref{lem-lemlaw11}. It belongs to $L^2(\Omega', H^1(M))$.
\end{remark}

\begin{lemma}\label{lem-lemlaw33} The law of $X'(t)$ is the same as the law of $X(t)$. \end{lemma}

\begin{lemma}\label{lem-lawlaw} Let $X_1$ and $X_2$ be solutions to \eqref{eqonrv} written $X_i = Y_i + Z_i$ with initial datum $X_{1,0} $ and $X_{2,0}$ which have the same law. The random variables $Y_i$ satisfy $Y_i(t) = e^{-it(m-\lap)}Y_{0,i}$ with $Y_{0,i}$ a random variable with the same law as $Y_f(t=0)$. Then, for all $t$, $X_1(t)$ and $X_2(t)$ have the same law.\end{lemma}

\begin{proof}Thanks to Lemma \ref{lem-lemlaw33} and using the same notations, we can consider the random variables $X'_1$ and $X'_2$ instead of $X_1$ and $X_2$. Let $\varphi$ be the map from $\Omega_1'$ to $\Omega_2'$ defined as 
$$
\varphi((Y_{0,1},Z_{0,1})^{-1}(\{(u_0,u_1)\}) )= (Y_{0,2},Z_{0,2})^{-1}(\{(u_0,u_1)\})
$$
for all $u_1 \in H^1(M)$ and $u_0$ in the image of $Y_f(t=0)$. 

By construction, $Z_{1,0}'  = Z_{2,0}' \circ\varphi$ and $P'_2$  is the image measure of $P_1'$ under $\varphi$. 

By uniqueness of the flow of \eqref{eqonZ}, $Z'_2(t) \circ \varphi = Z'_1(t)$ and since $\varphi$ preserves the measure the law of $X'_2(t)$ is the same has the one of $X'_1(t)$. Therefore, thanks to Lemma \ref{lem-lemlaw33}, $X_1(t)$ and $X_2(t)$ have the same law.
\end{proof}

\subsubsection{Gaussian variables}

In this subsection, we prove that if $X_0$ is a Gaussian variable, then so is $X(t)$ the solution of \eqref{eqonrv} with initial datum $X_0$. What is more, we prove that if $\gamma$ is a solution of \eqref{eqonop} then there exists a Gaussian variable with covariance $\gamma$ that is a solution of \eqref{eqonrv}.

\begin{lemma}\label{lem-law22} Let $X_0$ be a Gaussian process of covariance $\gamma_0$ such that there exists a square root $\gamma_0^{1/2}$ of $\gamma$ satisfying $\textrm{Tr}((\gamma_0^{1/2}-\gamma_f^{1/2})^*(1-\lap)(\gamma_0^{1/2}-\gamma_f^{1/2})) < \infty$. Let $X(t)$ be the solution of \eqref{eqonrv} with initial datum $X_0$ then $X(t)$ is a Gaussian process.
\end{lemma}

\begin{proof} Write $\varphi(t,x) = \E(|X(t,x)|^2)$. We have $\varphi \in m + \mathcal C(\R,L^2(\R^d)) + \mathcal C(\R,L^1(\R^d))$. Write $Q_0 = \gamma_0^{1/2}-\gamma_f^{1/2}$ and $W_0 = \int e^{ikx} dW(k)$. We have 
$$
\textrm{Tr }(Q_0^*(1-\lap)Q_0) <\infty.
$$
Hence $Q_0^* (1-\lap) Q_0$ can be diagonalised into $Q_0^*(1-\lap) Q_0 = \sum_n \alpha_n |u_n \times u_n|$ with $u_n$ orthonormal in $L^2$, $\alpha_n\geq 0$ and $\sum_n \alpha_n =  \textrm{Tr }(Q_0^*(1-\lap)Q_0) <\infty$. After some manipulations of the expression, we have that
$$
\|Q_0W_0\|_{L^2(\Omega, H^1(\R^d))}^2 = \int \int |(1-\lap)^{1/2} Q_0 e^{ikx}|^2 dx dk.
$$
By using the decomposition of $Q_0^* (1-\lap) Q_0$ we get
$$
\|QW_0\|_{L^2(\Omega, H^1(\R^d))}^2 = \int \sum_n \alpha_n |\hat u_n (k) |^2 = \sum_n \alpha_n <\infty
$$
where $\hat u_n$ is the Fourier transform of $u_n$. Hence $Z_0 = Q_0W_0$ belongs to $L^2(\Omega, H^1(\R^d))$.

We derive an equation on $Q$ assuming that $X$ is equal to $X(t) = (Q(t) + e^{-it(m-\lap)} \gamma_f^{1/2} )W_0 = Q(t) W_0 + Y_f$. We get 
$$
i\partial_t Q = (\varphi - \lap) Q + (\varphi - m ) e^{-it(m-\lap)} \gamma_f^{1/2} .
$$
Finally, write $V = e^{it(m-\lap) }Q$, we get
$$
i\partial_t V = e^{it(m-\lap) }( \varphi -m) e^{-it(m-\lap) } (V + \gamma_f^{1/2})
$$
This equation is at least locally well-posed in $\mathcal L (H^1)$ for instance (this is why we require $f$ bounded) and the solution satisfies
$$
e^{-it(m-\lap)} V W_0 \in L^2(\Omega, H^1(\R^d)).
$$

By uniqueness of the solution of \eqref{eqonZ} we get that $X(t) = (Q(t) + e^{-it(m-\lap)} \gamma_f^{1/2} )W_0$ first locally in time and then globally. This ensures that $X(t)$ is a Gaussian variable (of covariance $(Q(t) + e^{-it(m-\lap)} \gamma_f^{1/2} )^2$).

\end{proof}

\begin{definition} Let $\Sigma_f$ be the set of non negative operators $\gamma_0$ such that there exists a square root of $\gamma_0$, $\gamma_0^{1/2}$ such that $\textrm{Tr}((\gamma_0^{1/2}-\gamma_f^{1/2})^*(1-\lap)(\gamma_0^{1/2}-\gamma_f^{1/2}))$ is finite. Let $d_f$ be the distance on this set defined as 
$$
d_f(\gamma_1,\gamma_2) = d_2(\nu_1,\nu_2)
$$
where $d_2$ is the Wasserstein distance defined in Remark \ref{rem-wass} and $\nu_i$ is the law of the Gaussian process with covariance $\gamma_i$. \end{definition}

\begin{remark} This distance is well-defined. Let $W_0 = \int e^{ikx}dW(k)$ where $W$ is defined as in the definition of $Y_f$. We have that $X_i = \gamma_i^{1/2} W_0$ is a Gaussian variable with covariance operator $\gamma_i$. Hence
$$
d_f(\gamma_1,\gamma_2) \leq \|X_1-X_2\|_{L^2(\Omega,H^1(M))}
$$
and $X_1 - X_2 = Z_1 - Z_2$ with $Z_i = (\gamma_i^{1/2} - \gamma_f^{1/2})W_0 \in L^2(\Omega, H^1(M))$.
\end{remark}

\begin{lemma}\label{lem-gau11} Assume that $\gamma \in \mathcal C(\R,\Sigma_f)$ is a solution to \eqref{eqonop}. Then there exists a probability space $\Omega$ and a Gaussian variable $X \in Y+\mathcal C(\R,L^2(\Omega,H^1(M)))$ solution of \eqref{eqonrv} and of covariance $\gamma$. \end{lemma}

\begin{proof} Let $\varphi = \rho_{\gamma}$ and let $X$ be the solution to 
$$
i\partial_t X = (-\lap + \varphi) X
$$
with initial datum $X_0$ a Gaussian variable with covariance $\gamma_0$. Its covariance $\gamma_X$ is the unique solution in $ \gamma_f + \mathcal L (H^1))$ to the linear equation 
\begin{equation}\label{eqonoplin}
i\partial_t \gamma_X = [-\lap + \varphi, \gamma_X]
\end{equation}
with initial datum $\gamma_{X_0} = \gamma(t=0)$.

Indeed, if an operator $\tilde \gamma$ belongs to $\Sigma_f$ then it also belongs to $\gamma_f + \mathcal L(H^1)$.

What is more, if $\gamma_1$ and $\gamma_2$ are two solutions of \eqref{eqonoplin} in $\gamma_f + \mathcal L(H^1)$ with the same initial datum then $\gamma_1 - \gamma_2$ is a solution to \eqref{eqonoplin} in $\mathcal L(H^1)$ with initial datum $0$.

Let $U(t)$ the flow of the linear equation in $H^1$ on $u$ : $i\partial_t u = (-\lap + \varphi) u$. The map $U$ is an invertible. Hence, since,
$$
i\partial_t(U(t)^* (\gamma_1 - \gamma_2) U(t)) = 0
$$
the equation $i\partial_t \gamma_X = [-\lap + \varphi, \gamma_X]$ has a unique solution. 

Since $\gamma$ is also a solution to $i\partial_t \gamma_X = [-\lap + \varphi, \gamma_X]$ with initial datum $\gamma(t=0)$ we get that $\gamma_X = \gamma$, thus $\E(|X|^2) = \rho_\gamma $, which ensures that $X$ is a solution to \eqref{eqonrv} and thanks to \ref{lem-law22}, a Gaussian solution to \eqref{eqonrv}.
\end{proof}

\subsubsection{Global well-posedness}

\begin{corollary}[of Proposition \ref{prop-gwp}] \label{cor-gwp}Let $\Psi_f$ be the map from $\Sigma_f$ to $\mathcal C( \R, \Sigma_f)$ such that $\Psi_f(t) (\gamma_0) = \gamma_{Y_f(t) + Z(t)}$ where $Z(t)$ is the solution to \eqref{eqonZ} with initial datum $Z_0$ the Gaussian random process with covariance operator $(\gamma_0^{1/2} - \gamma_f^{1/2})^2$. The map $\Psi_f$ is well-defined, it defines a solution to \eqref{eqonop} and it is continuous for the distance $d_f$. Besides, $\Psi_f(t)\gamma_0$ is the unique solution to \eqref{eqonop} with initial datum $\gamma_0$.
\end{corollary}

\begin{proof} We take $Z_0 =  (\gamma_0^{1/2}-\gamma_f) W_0$. It belongs to $L^2(\Omega, H^1(M))$ and $X_0 = Y_f(t=0) + Z_0$. Thanks to Proposition \ref{prop-gwp} we get a solution $X$ of \eqref{eqonrv}, and thanks to Proposition \ref{XtoG}, we get a solution $\gamma$ to the equation \eqref{eqonop}. Hence $\Psi_f(t)$ is well-defined. It is continuous in time and in the initial datum for the following reason : the distance between $\gamma_{X_1(t)}$ and $\gamma_{X_2(t)}$ is controlled by the norm of $X_1(t) - X_2(t)$, which is controlled by the norm of $Z_1(t)-Z_2(t)$. The continuity of the solution $Z(t)$ in both time and initial datum gives the result. 

For the uniqueness of the solution, let $\gamma_1$ and $\gamma_2$ be two solutions of \eqref{eqonop} with the initial datum $\gamma_0$. For $i=1,2$, there exists $X_i(t)$ a solution of \eqref{eqonrv} which is a Gaussian variable of covariance $\gamma_i$. For $i = 1,2$, $X_i(t=0)$ is a Gaussian variable of covariance $\gamma_0$. Hence $X_1(t=0)$ and $X_2(t=0)$ have the same law. Therefore $X_1(t)$ and $X_2(t)$ too, which ensures that $\gamma_1 = \gamma_2$ and hence the uniqueness of the solution of \eqref{eqonop}.
\end{proof}

\begin{remark} One could rewrite the corollary \ref{cor-gwp} as : the equation \eqref{eqonQ} is globally well-posed in $\Sigma_f$.
\end{remark}

\subsection{On the focusing case}

We consider the equation 
\begin{equation}\label{eqonopfoc}
i\partial_t \gamma = [-\lap - \rho_\gamma,\gamma]
\end{equation}
on $\R^d$, $d=2,3$.

\begin{corollary}[of Proposition \ref{prop-blowup}]\label{cor-blowup} If $\gamma_0$ is such that $\textrm{Tr}((1-\lap) \gamma_0) < \infty$, $\int_{\R^d} |x|^2 \rho_{\gamma_0}(x)dx < \infty$ and 
$$
\frac12 \textrm{Tr}((1-\lap) \gamma_0) - \frac14\int_{\R^d} \rho_{\gamma_0}(x)^2 <0
$$
then there is at least one solution of \eqref{eqonopfoc} that exists locally in time and blows up at finite time in the sense that there exists $T$ such that
$$
\textrm{Tr}((1-\lap) \gamma(t)) \rightarrow \infty
$$
when $t\rightarrow T$.
\end{corollary}

\begin{proof} This is due to the fact that with $X_0$ the Gaussian random field of covariance $\gamma_0$ and $X(t)$ the solution of \eqref{eqonrvfoc}, we have
$$
\textrm{Tr}((1-\lap) \gamma(t)) = \|X(t)\|_{L^2(\Omega, H^1(\R^d))}
$$
and
$$
\int_{\R^d} |x|^2 \rho_{\gamma_0}(x)dx = V(t=0)
$$
and
$$
\frac12 \textrm{Tr}((1-\lap) \gamma_0) - \frac14\int_{\R^d} \rho_{\gamma_0}(x)^2 = \mE(X_0).
$$
\end{proof}

\begin{remark} One could rewrite the corollary \ref{cor-gwp} as : there exist blow-up solutions to the equation \eqref{eqonopfoc}. \end{remark}

\section{Acknowledgements}
The author would like to thank Nikolay Tzvetkov for numerous and valuable discussions on the subject, which helped both with the presentation and contents of this paper.

\providecommand{\bysame}{\leavevmode\hbox to3em{\hrulefill}\thinspace}
\providecommand{\MR}{\relax\ifhmode\unskip\space\fi MR }
% \MRhref is called by the amsart/book/proc definition of \MR.
\providecommand{\MRhref}[2]{%
  \href{http://www.ams.org/mathscinet-getitem?mr=#1}{#2}
}
\providecommand{\href}[2]{#2}

\end{document}